\documentclass[final,a4paper,11pt]{article}
\usepackage{mathptmx}
\usepackage{amssymb,amsmath, amsfonts,amsthm}
\usepackage{a4wide}
\usepackage{color}

\newcommand{\R}{\mathbb{R}}

\newcommand{\norm}[1]{\|#1\|}
\newcommand{\bnorm}[1]{\big\|#1\big\|}

\newcommand{\dist}[1]{{\rm dist}(#1)}

\newcommand{\mv}{\,\mid\,}

\newcommand{\Np}{\hat{\cal N}^p}

\newcommand{\skalp}[1]{\langle #1\rangle}

\newcommand{\xb}{\bar x}

\newcommand{\zb}{\bar z}

\newcommand{\oo}{o}

\newcommand{\lin}[1]{{\cal L}(#1)}

\newcommand{\cl}{{\rm cl\,}}

\newcommand{\co}{{\rm conv\,}}
\newcommand{\gph}{\mathrm{gph}\,}
\newcommand{\dom}{\mathrm{dom}\,}
\newcommand{\tto}{\rightrightarrows}

\newcommand{\myvec}[1]{\left(\begin{array}{c}#1\end{array}\right)}

\if{

}\fi

\newtheorem{theorem}{Theorem}[section]
\newtheorem{proposition}[theorem]{Proposition}
\newtheorem{remark}[theorem]{Remark}
\newtheorem{lemma}[theorem]{Lemma}
\newtheorem{corollary}[theorem]{Corollary}
\newtheorem{definition}[theorem]{Definition}

\title{Second-order optimality conditions for general nonconvex optimization problems and variational analysis of disjunctive systems}
\author{Mat\'u\v{s} Benko\thanks{Applied Mathematics and Optimization, University of Vienna, 1090 Vienna, Austria,
e-mail: matus.benko@univie.ac.at. This author's research was supported by the Austrian Science Fund (FWF) under grant P32832-N.}
\and Helmut Gfrerer\thanks{Institute of Computational Mathematics, Johannes Kepler University Linz, Austria, e-mail: helmut.gfrerer@jku.at. }
\and Jane J. Ye\thanks{Department of Mathematics
and Statistics, University of Victoria, Victoria, B.C., Canada V8W 2Y2, e-mail: janeye@uvic.ca. The research of this author was partially
supported by NSERC.}
\and Jin Zhang\thanks{Department of Mathematics, Southern University of Science and Technology, National Center for Applied Mathematics Shenzhen, Shenzhen, P.R. China.
e-mail: zhangj9@sustech.edu.cn.  This author's work is supported by National Natural Science Foundation of China (12222106), Guangdong Basic and Applied Basic Research Foundation (2022B1515020082).}
\and Jinchuan Zhou\thanks{Department of Statistics, School of Mathematics and Statistics, Shandong University of Technology,
 Zibo 255049, P.R. China, e-mail: jinchuanzhou@163.com. This author's work is supported by National Natural Science Foundation of China (11771255, 12271309), Young Innovation Teams of Shandong Province (2019KJI013) and Shandong Province
Natural Science Foundation (ZR2021MA066).}}

\date{}

\begin{document}

\maketitle

 \noindent
 {\small
 {\bf Abstract.}\ In this paper, we propose second-order sufficient optimality conditions
 for a very general nonconvex constrained optimization problem,
 which covers many prominent mathematical programs.
 Unlike the existing results in the literature, our conditions prove to be sufficient,
 for an essential local minimizer of second order,
 under merely basic smoothness and closedness assumptions on the data defining the {problem}.
 In the second part, we propose a comprehensive first- and second-order variational analysis of disjunctive systems and
 demonstrate how the second-order objects appearing in the optimality conditions can be effectively computed in this case.

}

\section{Introduction}

For decades, variational analysis has been recognized as an important tool for studying optimization problems; we refer to the standard monographs \cite{BonSh00,ClarkeLSW,Io17,Mor,Mornew,RoWe98}.
Recently, second-order variational analysis has been developed rapidly;
see \cite{GfrMo17,GfrOut16a,GfrYeZh19,MoMoSa20,Ro2021a,Ro2021b} and the references therein.

In this paper, we will deal with some special aspects of second-order  variational analysis, namely second-order optimality conditions for an optimization problem in the form:
\begin{flalign*}
\begin{split}
\mbox{(GP)} \hspace{49mm} \min & \ \  f(x)\ \ \ {\rm s.t. } \ \  g(x)\in C.
\end{split}&
\end{flalign*}
Here $C \subset \mathbb{R}^m$ is a closed set and $f:\mathbb{R}^n\rightarrow \mathbb{R}$ and $g:\mathbb{R}^n\rightarrow \mathbb{R}^m$ are twice continuously differentiable functions unless otherwise specified.
This general model covers many common optimization problems, including the very challenging ones with constraints expressed via complementarity relations,
in which case not only the feasible
set $g^{-1}(C):= \{x \in \R^n \mv g(x) \in C\}$ but also
the set $C$ is nonconvex, see the comments below. 
\if{\textcolor{blue}{Although for simplicity of notation, we assume the spaces considered are Euclidean spaces, our analysis can be carried over to finite dimensional Hilbert spaces, and so it can cover semidefinite complementarity programs.}
}\fi

We concentrate on the development of tight second-order optimality conditions, i.e., the difference between the necessary and sufficient conditions should be small.
Note that there are also other intrinsic issues of second-order conditions like stability of solutions or the convergence of numerical algorithms.  For
example, Rockafellar \cite{Ro2021a, Ro2021b} has demonstrated the importance of second-order variational analysis in numerical optimization, but these topics are far beyond the scope of this paper.

Let us now provide a brief discussion on existing results dealing with second-order optimality conditions, both necessary and sufficient.
If $C$ is convex polyhedral, as in case of the standard nonlinear programs, second-order optimality conditions can be expressed via second derivative of the Lagrangian.
If $C$ lacks polyhedrality, however, an additional term is needed to capture the curvature of $C$ and there are various tools that can be utilized for that purpose.
When $C$ is convex, a comprehensive analysis of second-order conditions is available in Bonnans and Shapiro \cite[Sections 3.2 and 3.3]{BonSh00}. There, the second-order necessary conditions are derived within the framework of convex analysis and are of the following form, cf. \cite[Theorem 3.45]{BonSh00}: If a suitable constraint qualification holds at a local minimizer $\xb$ then for every critical direction $u$ and every convex subset $K(u)$ of the second-order tangent set $T_C^2(g(\xb); \nabla g(\xb)u)$ there is a multiplier $\lambda$ fulfilling first-order optimality conditions such that
\begin{equation}
   \nabla^2_{xx}L(\bar{x},\lambda)(u,u) -{\sigma}_{K(u)}(\lambda) \geq 0.\label{sonc1-support}
\end{equation}
Here, $L$ denotes the Lagrangian and $\sigma$ is the support function. In particular, if the second-order tangent set $T_C^2(g(\xb); \nabla g(\xb)u)$ is convex, we arrive at the condition
\begin{equation}
   \nabla^2_{xx}L(\bar{x},\lambda)(u,u) -{\sigma}_{T_C^2(g(\xb); \nabla g(\xb)u)}(\lambda) \geq0.\label{sonc-support}
\end{equation}
By \cite[Proposition 3.46]{BonSh00}, this condition is also necessary at a local minimizer, provided the multiplier $\lambda$ fulfilling the first-order optimality condition is unique,  regardless whether or not  $T_C^2(g(\xb); \nabla g(\xb)u)$ is convex.

In the very recent paper by Gfrerer et al.~\cite{GfrYeZh19}, optimality conditions have been stated for nonconvex $C$. In this case one has to consider different types of first-order optimality conditions involving strong (S-), Mordukhovich (M-) and Clarke (C-) multipliers, respectively. Moreover, the feasible region may behave quite differently when moving away from the minimizer $\xb$ in different directions. This fact motivates the use of different constraint qualifications and different types of multipliers when considering different critical directions.

When  a directional non-degeneracy condition  for the critical direction $u$ is satisfied, ensuring that directional S-, M- and C-multipliers coincide and are unique, condition \eqref{sonc-support} remains valid, see \cite[Corollary 5]{GfrYeZh19}. When relaxing the directional non-degeneracy to the directional Robinson constraint qualification, one can still show that for every convex subset $K(u)$ of $T_C^2(g(\xb); \nabla g(\xb)u)$ there is some directional C-multiplier $\lambda$ satisfying \eqref{sonc1-support}, cf. \cite[Corollary 4]{GfrYeZh19}. As it is shown in \cite[Proposition 8]{GfrYeZh19}, this is a very strong second-order necessary  condition. However, it has the disadvantage that the directional C-multiplier $\lambda$ does not only depend on the critical direction $u$, but also on the convex set $K(u)$. This can be remediated by the use of the so-called lower generalized support function $\hat \sigma$. It was shown in \cite{GfrYeZh19} that under the directional metric subregularity constraint qualification, which is weaker than the directional Robinson constraint qualification,  there is a directional M-multiplier $\lambda$ such that
\begin{equation*}
  \nabla^2_{xx}L(\bar{x},\lambda)(u,u) -\hat{\sigma}_{T_C^2(g(\xb); \nabla g(\xb)u)}(\lambda)\geq 0. \label{nosc}
\end{equation*}
This function $\hat{\sigma}$ is indeed an extension of the support function as for any closed set $D$ we  have $\hat\sigma_D\leq\sigma_D$ and the two coincide when applied to a convex set $D$.

Now let us consider the second-order sufficient conditions. Let $L^\alpha(x,\lambda):=\alpha f(x)+\skalp{\lambda, g(x)}$. If, at a feasible point $\xb$, for every critical direction $u$ the  set $C$ is {\em outer second-order regular} at $g(\xb)$ in direction $\nabla g(\xb)u$ and there are $\alpha\geq 0$ and $\lambda\in\R^m$  such that $\alpha \nabla f(\xb)u=0$, $\nabla_x L^\alpha(\bar{x},\lambda)=0$ and
\begin{equation*}
   \nabla^2_{xx}L^\alpha(\bar{x},\lambda)(u,u) -{\sigma}_{T_C^2(g(\xb); \nabla g(\xb)u)}(\lambda) >0 \label{sosc-support}
\end{equation*}
then the point $\xb$ is a local minimizer fulfilling the so-called {\em quadratic growth condition}. In a slightly different form, this result was first proved in \cite[Theorem 3.86]{BonSh00}  for convex sets $C$ and then extended in \cite[Theorem 4]{GfrYeZh19} to the nonconvex case. These sufficient conditions were essentially improved in the recent work by Mohammadi et al. \cite[Proposition 7.3]{MoMoSa20}, where the assumption of outer second-order regularity is dropped and the sigma term  $-{\sigma}_{T_C^2(g(\xb); \nabla g(\xb)u)}(\lambda)$ is replaced by the second-order subderivative ${\rm d}^2\delta_{C}(g(\bar{x});\lambda)(\nabla g(\bar{x})u)$ of the indicator function $\delta_C$. One of the main results of this paper is an improvement of \cite[Proposition 7.3]{MoMoSa20} in that the set $C$ is neither assumed to be convex nor to be parabolically derivable. Moreover, we do not need  the existence of an S-multiplier and we can choose different multipliers for every critical direction in order to fulfill the second-order sufficient condition. Finally, we  not only prove the quadratic growth condition, but  also show that the point in question is an essential local minimizer of second order.

Summing up these considerations, we see that, besides the imposed constraint qualification, the second-order optimality conditions rely on the three second-order objects $\hat\sigma_{T^2_C(\bar z;w)}(\lambda)$, $\sigma_{T^2_C(\bar z;w)}(\lambda)$ and ${\rm d}^2\delta_{C}(\bar z;\lambda)(w)$, each of them describing in some way the curvature of the set $C$, and which are linked together by the inequalities
\begin{equation*}\label{3_curv_terms_rel}
{\rm d}^2\delta_C(\zb;\lambda)(w)\leq -\sigma_{T^2_C(\zb;w)}(\lambda)\leq -\hat{\sigma}_{T^2_C(\zb;w)}(\lambda),
\end{equation*}
valid for any closed set $C$, every tangent $w\in T_C(\zb)$ and every $\lambda$ with $\skalp{\lambda,w}\geq 0$, see Proposition \ref{LemBasicSecOrdSubDer} below.

{When studying optimization problems via some elaborate machinery of variational analysis, interesting as it may be, the important question remains:
Can the employed tools be effectively computed or estimated and the obtained results suitably applied?}

While with convex $C$ we can formulate several standard programs as problem (GP), such as nonlinear programs, second-order cone programs, etc., we are primarily interested in programs modeled with nonconvex $C$.
Such programs are considered very challenging but also increasingly important by the optimization community.
Among others, they include the bilevel programs (see, e.g., Dempe \cite{dempe} and Ye and Zhu \cite{YeZhu}), programs with constraints governed by quasi-variational inequalities (see, e.g., Mordukhovich and Outrata \cite{MoOut07}),
and the mathematical program with second-order cone complementarity constraints (SOC-MPCC) (see, e.g., Outrata and Sun \cite{OutSun} and Ye and Zhou \cite{YeZhou}).
All of these problem classes can be modeled as a special case of problem (GP) with set $C$ possessing the following structure:
\begin{equation}\label{eq:StructureOfC}
C=\{(y,b(y)^T\eta)\mv (q(y),\eta)\in\gph N_P\},
\end{equation}
where $b$ and $q$ are sufficiently smooth mappings ($b$ maps into the space of matrices of an appropriate dimension{, in many cases there holds $b=\nabla q$}), $P$ is a convex polyhedral set and $N_P$ is the associated normal cone mapping.
{By taking $z:=(y, \eta)$, $B(z_1,z_2):=(z_1, b(z_1)^Tz_2)$,  $ G(z_1,z_2):=(q(z_1),z_2)$ and $D:=\gph N_P,$ the set $C$ given by \eqref{eq:StructureOfC} can be represented as}
\if{Take $ x:=(z,\eta)$,  $G(x):=(z, -\nabla_y f(z))$ and  $$C:=\{(z, \nabla_y q(z)^T \eta)| (q(z),\eta )\in \gph N_P\}.$$ By taking $z:=(y, \eta)$, $B(z_1,z_2):=(z_1, b(z_1)^Tz_2)$,  $ G(z_1,z_2):=(q(z_1),z_2)$ and $D:=\gph N_P,$ the set $C$ given by \eqref{eq:StructureOfC} can be represented as}
\fi
$$C=B(\Gamma), \mbox{ where } \Gamma=\{(z_1,z_2)\mv G(z_1,z_2)\in D\}.$$
Since $P$ is assumed to be convex polyhedral, the set $D$ is polyhedral, i.e., $D$ is the finite union of convex polyhedral sets.


{ In this paper, we will calculate $ \hat{\sigma}_{T_\Gamma^2(\bar z; w)}(\lambda)$, ${\sigma}_{T_\Gamma^2(\bar z; w)}(\lambda)$ and ${\rm d}^2\delta_{\Gamma}(\bar z;\lambda)(w)$ and defer the calculation of these three quantities for $\Gamma$ replaced by the set $C$ defined by \eqref{eq:StructureOfC} to a forthcoming paper Benko et al. \cite{BeGfrYeZhouZhang}. To accomplish this goal, in the second part of the paper, we  provide} a comprehensive first- and second-order variational analysis of the disjunctive system
\begin{equation*}
{\Gamma}:=G^{-1}(D) = \{ x \in \mathbb{R}^n \mv {G}(x)\in {D}\},\label{DJ}
\end{equation*}
where ${G}: \R^n \to \R^d$ is twice continuously differentiable and ${D} \subset \R^d$ is assumed to be polyhedral. The obtained results are also of independent interests and they may be useful in other applications.

We organize our paper as follows. Section 2 contains the preliminaries and auxiliary results.  In Section 3,  we derive second-order sufficient optimality condition for the general program (GP).  Sections 4 and 5 are devoted to the first- and the second-order variational analysis of the set $G^{-1}(D)$, respectively.  Finally, in Section 6 we demonstrate how to recover the second-order necessary and sufficient conditions from  \cite{Gfr14a} for the disjunctive program by means of our results.

\section{Preliminaries and auxiliary results}
In this section, we recall some background material from variational analysis and provide some preliminary results.
Let us begin with the notation.
The open unit ball is denoted by $\mathbb{B}$ and the open ball centred at $z$ with radius $\delta$ is denoted by $ \mathbb{B}(z,\delta)$.
For a set $S \subset \R^n$, denote by span $S$, cl $S$ and conv $S$,
its linear span, closure and convex hull, respectively.
We call a subspace $L \subset \R^n$ the {\em generalized lineality space} of $S$ and denote it by $\lin{S}$ provided that it is the largest subspace satisfying $S+L\subset S$.
Since any linear subspace includes $0$, we actually have $S+\lin{S}=S$,
and if $S$ is a closed convex cone, we get $\lin{S}=S\cap (-S)$.
The indicator function $\delta_S:\R^n \to \bar \R := [-\infty,+\infty]$ of $S$
is given as $\delta_S(z)=0$ for $z \in S$ and $\delta_S(z)=+\infty$ if $z \notin S$.
Next, if $S$ is closed, let $S^\circ$ and $\sigma_S :\R^n \to \bar \R$ stand for the polar cone to $S$ and the support function of $S$, respectively, i.e.,
 $S^\circ:=\{z^* \in \R^n \mv \langle z^*,z \rangle \leq 0,\  \forall z\in S \}$
and $\sigma_S(z^*):=\sup \{\langle z^*,z \rangle \mv z\in S\}$ for $z^*\in \R^n$.
{For an extended function $\varphi:\R^n \to \bar \R$, we define its effective domain by ${\rm dom}\varphi:=\{z|\varphi(z) <+\infty\}$}. For $w \in \R^n$, denote by ${\{w\}}^\perp$ the orthogonal complement of the linear space generated by $w$.
 Let $o:\mathbb{R}_+\rightarrow \mathbb{R}^n$ stand for a mapping with the property that $o(t)/t \rightarrow 0$ when $t \downarrow 0$.
 The symbol $z^{\prime} \overset{S}{\to} z$ indicates that $z^{\prime}\in S$
 and $z^{\prime} \rightarrow z$.
  For a mapping $\psi:\mathbb{R}^n\to \mathbb{R}^d$ and $z \in \R^n$, we denote by
 $\nabla \psi(z)\in \mathbb{R}^{d\times n}$ its Jacobian at $z$
 and by $\nabla^2 \psi(z)$ its second derivative at $z$ as defined by
 $$w^T\nabla ^2 \psi(z) :=\lim_{t\rightarrow 0} \frac{\nabla \psi(z+ tw)-\nabla \psi( z)}{t} \quad \forall w \in \mathbb{R}^n.$$
 Hence
 $$\nabla^2 \psi(z)(w,w):=w^T \nabla^2 \psi(z) w=(w^T \nabla^2 \psi_1(z)w, \dots, w^T \nabla^2 \psi_d(z)w)^T \quad \forall w\in \mathbb{R}^n.$$

\subsection{Variational geometry}

First we review the various classical concepts of tangent and normal cones.
 \begin{definition}[Tangent and normal cones, \cite{RoWe98}]
 Given $S\subset \mathbb{R}^n$, $z\in S$, the
 regular/Clarke tangent cone and
 tangent/contingent cone to $S$ at $z$ are defined, respectively, by
 \begin{eqnarray*}
  \widehat{T}_S(z)&:=&
  \liminf\limits_{{z' \stackrel{S}{\to}z} \atop {t\downarrow 0}}\frac{S-z'}{t}=
  \Big\{ w\in \R^n\,\Big|\, \forall \, t_k\downarrow 0,\, {z_k \stackrel{S}{\to}z}, \; \exists w_k\to w \
  \ {\rm with}\ \ z_k+t_k w_k\in S  \Big\},\\
 T_S(z)&:=&
  \big\{w \in\R^n \, \big| \, \exists \ t_k\downarrow 0,\; w_k\to w \ \ {\rm with}
 \ \ z+t_k w_k\in S \big\}.
 \end{eqnarray*}
For $w\in T_S(z)$,
the outer second-order tangent set to $S$ at $z$ in direction $w$ is defined by
  \begin{eqnarray*}
 T_S^2(z; w)
 := \left \{ s \in \R^n \, \Big| \, \exists \ t_k \downarrow 0, s_k \rightarrow s \ \ {\rm with}
 \ \  z+t_kw+\frac{1}{2}t^2_k s_k \in S
  \right \}.
 \end{eqnarray*}
 The regular/Fr\'echet normal cone, the proximal normal cone, and the limiting/Mordukhovich normal cone to $S$ at $z$ are given, respectively, by
 \begin{eqnarray*}
 \widehat{N}_S(z)&:=&\left\{z^*\in \R^n \, \big| \, \langle z^*, z'-z\rangle \le
 o\big(\|z'-z\|\big), \ \forall z'\in S\right\},\\
 \widehat{N}^p_S(z)&:=&\left\{z^* \in \R^n \, \big| \, \exists \ \gamma>0:   \skalp{z^*,z'-z}\leq \gamma \|z'-z\|^2, \ \ \forall z'\in S \right\},\\
 N_S(z)&:=&
 \left \{z^* \in \R^n \, \Big| \, \exists \ z_k \overset{S}{\to} z, z^*_{k}\rightarrow z^* \ {\rm with }\  z^*_{k}\in \widehat{N}_{S}(z_k) \right \}.
\end{eqnarray*}
 \end{definition}
 Recall that for any set $S$, one always has $\widehat{N}_S(z)=T_S(z)^\circ$ and if $S$ is closed,  then
 ${N}_S(z)^\circ=\widehat T_S(z)$; cf. Rockafellar and Wets \cite[Theorem 6.28]{RoWe98}.

 Recently, motivated by the formula $\widehat{T}_S(z)=
  \liminf_{z' \stackrel{S}{\to}z}T_S(z')$ (cf. \cite[Theorem 6.26]{RoWe98}),
a directional variant of the regular tangent cone has been introduced.
  \begin{definition}[Directional regular tangent  cone, {\cite[Definition 2]{GfrYeZh19}}]
  Given $S\subset \mathbb{R}^n$, $z\in S$ and $w\in \R^n$, the regular/Clarke tangent cone to $S$ at $z$ in direction $w$ is defined by
 \begin{eqnarray*}
 \widehat{T}_{S}(z;w):=
  \liminf\limits_{{t\downarrow 0, w'\to w}\atop {z+tw'\in S}}T_S(z+tw')
=\Big\{v\in \R^n \big| \, \forall t_k\downarrow 0, w_k \to w,   z+t_k w_k \in S, \exists v_k\rightarrow v \ {\rm with }\ v_k\in T_S( z+t_k w_k) \Big\}.
 \end{eqnarray*}
 \end{definition}
 It is easy to see that $\widehat{T}_{S}(z;0)=\widehat{T}_{S}(z)$.
 Similar to the regular tangent cone,  the directional regular tangent cone $\widehat{T}_S(z;w)$ is a closed and convex cone, see \cite[Proposition 3]{GfrYeZh19}.
\begin{proposition}[{\cite[Proposition 1]{GfrYeZh19}}]\label{PropRegTanCone}
 Given a closed set $S\subset \mathbb{R}^n$, for every $z\in S$ and every
$w\in T_S(z)$ one has
\[T_{T_S(z)}(w)+\widehat T_S(z;w)=T_{T_S(z)}(w),\quad T_S^2(z; w)+\widehat T_S(z;w) = T_S^2(z; w).\]
\end{proposition}

 \begin{definition}[Directional normal cones, {\cite{Gfr13a,GfrYeZh19,GM}}]\label{NormalCone}
 Given $S\subset \mathbb{R}^n$, $z\in S$ and a direction $w\in \mathbb{R}^{n}$, the limiting and the Clarke normal cone to $S$ in  direction $w$ at $z$ are given, respectively, by
 \begin{eqnarray*}
N_{S}(z; w)&:=&
\left \{z^* \in \R^n \, \Big| \, \exists \ t_{k}\downarrow 0, w_{k}\rightarrow w, z^*_{k}\rightarrow z^* \ {\rm with }\  z^*_{k}\in \widehat{N}_{S}(z+ t_{k}w_{k}) \right \},\\
N_{S}^c(z; w)&:=& {
{\cl} \co} N_{S}(z; w).
\end{eqnarray*}
\end{definition}

By \cite[Lemma 2.1]{Gfr14a}, when $S$ is the union of finitely many convex polyhedral sets, we have that for any $w\in T_S(z)$,
\begin{equation*}\label{dir-poly-1}
N_S(z;w)\subset N_S(z)\cap \{w\}^\perp.
\end{equation*}
Moreover, if $S$ is a closed convex set and $w\in T_S(z)$,
\begin{equation}
N_S(z;w)=N_{T_S(z)}(w)=N_S(z)\cap \{w\}^\perp.\label{convexcase}
\end{equation}

 \begin{proposition}[Directional tangent-normal polarity, {\cite[Proposition 3]{GfrYeZh19}}] \label{polarity}
 For a closed set $S\subset \mathbb{R}^n$, $ z\in S$, and $w\in \mathbb{R}^n$, one has
 $$\widehat{T}_S(z;w)=N_S(z;w)^\circ=N_S^c(z;w)^\circ, \quad \widehat{T}_S(z;w)^\circ=N_S^c(z;w).$$
 \end{proposition}

 The definition of the lineality space $\lin{S}$ readily yields
\begin{equation}\label{EqConsequLin}
T_S(z+l)=T_S(z),\quad \widehat N_S(z+l)=\widehat N_S(z),\quad \forall z\in S, \quad \forall l\in \lin{S}.
\end{equation}
 By the previous proposition, we also get the following result.
\begin{proposition}\label{lem3.1}
Let $S\subset \mathbb{R}^n$ be a closed set, $z\in S$ and $w\in T_S(z)$. Then
\begin{equation}\label{Eq18}
\big({\rm span\,}{N_S(z; w)}\big)^\circ=\big({\rm span\,}{ N^c_S(z; w)}\big)^\circ=\lin{\widehat T_S(z; w)}
\subset \lin{T_{T_S(z)}(w)}.
\end{equation}
\end{proposition}
\begin{proof}
Notice that
 \begin{eqnarray*}
 \big({\rm span\,}{N_S(z; w)}\big)^\circ&=&\big({\rm span\,}{ N_S^c(z; w)}\big)^\circ
 =\big(N_S^c(z; w)- N_S^c(z; w)\big)^\circ
 = \big( N_S^c(z; w) \big )^\circ \cap  -\big( N_S^c(z; w) \big)^\circ \\
 &=&\lin{ N_S^c(z; w)^\circ} =\lin{\widehat T_S(z; w)},
 \end{eqnarray*}
where the first equality holds obviously using the fact that the set $ N_S^c(z; w)$ is  a closed convex cone, the second equality follows from \cite[Theorem 2.7]{Roconvexanalysis},
the third equality follows from the calculus rule for polar cones in \cite[Corollary 16.4.2]{Roconvexanalysis}, the fourth equality hold by the fact that the set $N_S^c(z;w)^\circ$ is a closed convex cone and the fifth equality holds by Proposition \ref{polarity}.

Proposition \ref{PropRegTanCone} yields
$T_{T_S(z)}(w) + \widehat T_S(z; w) = T_{T_S(z)}(w)$ and since
$\widehat T_S(z; w)$ is a closed convex cone by
\cite[Proposition 3]{GfrYeZh19},
we have
$\lin{\widehat T_S(z; w)} = \widehat T_S(z; w) \cap (-\widehat T_S(z; w))$.
Thus
\[
    T_{T_S(z)}(w) + \lin{\widehat T_S(z; w)} \subset T_{T_S(z)}(w)
\]
holds as well and the inclusion in \eqref{Eq18} follows by the definition of the lineality space.
\end{proof}


\subsection{Directional proximal normal cone}

It turns out that we also need a directional version of the proximal normal cone, see Proposition \ref{LemBasicSecOrdSubDer}.
To this end, we need the following definition.

\begin{definition}[Directional neighborhood, {\cite{Gfr13a}}] \label{DirectionN}
Let $w \in \mathbb{R}^n$. For $\delta, \rho>0$,
$$
 V_{\delta, \rho}(w) := \left \{ w'\in \delta \mathbb{B}\left |
 \big\| \| w \|w' -\|w'\| w \big\|\leq \rho \|w'\|\| w \| \right. \right\}
$$
  is called a directional neighborhood of direction $w$.
\end{definition}
It is easy to see that $V_{\delta,\rho}(w)\subset V_{\delta,\rho}(0)=\delta \mathbb{B}$.
Hence, the directional neighborhood is in general smaller than the classical neighborhood.

{Recall that the proximal normal cone to a closed set $S$ at a point $z\in S$ can be equivalently given by
\[\widehat{N}^p_S(z):=\{z^* \in \R^n \mv \exists \ \delta, \gamma>0:   \skalp{z^*,z'-z}\leq \gamma \|z'-z\|^2 \ \ \forall z'\in S\cap (z+\delta \mathbb{B})\};\]
see, e.g., \cite[Proposition 1.5]{ClarkeLSW}.
By replacing the standard neighborhood by the directional one, we arrive at the following directional version of the proximal normal cone.}

\begin{definition}[Directional proximal normal cone]
Given a closed set $S\subset\mathbb R^n$, a point $z \in S$ and a direction $w\in T_S(z)$, we define the proximal prenormal cone to $S$ in direction $w$ at $z$ as
\[\Np_S(z;w):=\{z^* \in \R^n \mv \exists \ \delta,\rho, \gamma>0: \skalp{z^*,z'-z}\leq \gamma \|z'-z\|^2 \ \ \forall z'\in S\cap (z+V_{\delta,\rho}(w))\} ,\]
and the proximal normal cone to $S$ at $z$ in direction $w$ as
\[\widehat N^p_S(z;w):= \Np_S(z;w)\cap \{w\}^\perp.\]
In case when $w\not\in T_S(z)$ we set $\Np_S(z;w):= \widehat N^p_S(z;w):=\emptyset$.
\end{definition}
From definition, we can see that the proximal prenormal cone is  in general larger than the classical proximal normal cone, i.e. $\widehat{N}_S^p(z) \subset \Np_S(z;w)$ for any $w\in T_S(z)$.
Moreover from definition, any vector $z^*$ satisfying $\langle z^*, w\rangle <0$ is always included in $\Np_S(z;w)$.
In fact we have
\begin{equation}
\{z^*| \langle z^*, w\rangle <0\} \subset \Np_S(z;w) \subset \{z^*| \langle z^*, w\rangle \leq 0\}. \label{prenormal} \end{equation}
Since the vectors $z^*$ satisfying  $\langle z^*, w\rangle <0$  do not provide much useful information, it is natural to restrict the directional proximal prenormals by intersecting with { the orthogonal complement of $w$}. This restriction yields the concept of directional proximal normal cone, and ensures that
the directional proximal normal is {contained} in the directional limiting normal cone, see Proposition \ref{PropDirProxNormalCone}.
In particular when $S$ is a closed convex set, {combining} (\ref{convexcase}) and (\ref{EqPropDirProxNormalCone}) below, we get
\begin{equation}\label{equality-convex}
\widehat N^p_S(z;w)=N_S(z;w)= N_S(z)\cap \{w\}^\perp=N_{T_S(z)}(w).
\end{equation}


In the following proposition we show convexity of the directional proximal normal cone and compare it with other normal cones.

\begin{proposition}\label{PropDirProxNormalCone}
Let $S\subset\R^n$ be closed and let $w\in T_S(z)$ be given. Then both $\Np_S(z;w)$ and $\widehat N^p_S(z;w)$ are convex cones and
\begin{equation}\label{EqPropDirProxNormalCone}
  \widehat N^p_S(z)\cap\{w\}^\perp\subset \widehat N^p_S(z;w)\subset \widehat N_{T_S(z)}(w){\subset N_{T_S(z)}(w) \subset N_S(z;w)}.
\end{equation}
\end{proposition}
\begin{proof}
By definition, it is easy to show that $\Np_S(z;w)$ is a convex cone. Thus $\widehat N^p_S(z;w)$ is also a convex cone as the intersection of two convex cones.

  The first inclusion in \eqref{EqPropDirProxNormalCone} follows immediately from $\widehat{N}_S^p(z) \subset \Np_S(z;w)$, the third one is trivial and the last one
  was proved in \cite[Lemma 3]{GfrYeZh19}.
  Thus, it remains to show the second inclusion.

  Since
  $$\widehat N^p_S(z;0)=\widehat N^p_S(z)\subset \widehat N_S(z)= \widehat N_{T_S(z)}(0),$$
  where the last equation follows from \cite[eq. (3)]{GfrYe20}, the inclusion holds true for $w=0$.
  Now let $w\not=0$ and consider $z^*\in \widehat N^p_S(z;w)$.
  We wish to show that
 \begin{equation}\label{Eq(8)}
 z^*\in  (T_{T_S(z)}(w))^\circ =\widehat N_{T_S(z)}(w).
 \end{equation}
  By definition, we can find some $\delta>0, \gamma>0$ such that
  \begin{equation}\label{EqAuxDirProx1}
  \skalp{z^*,z'-z}\leq \gamma\norm{z'-z}^2\quad \forall z'\in (z+V_{\delta,\delta}(w))\cap S.
  \end{equation}
To show (\ref{Eq(8)}) we pick $v\in T_{T_S(z)}(w)$ together with sequences $t_k\downarrow 0$ and $v_k\to v$ satisfying $w+t_kv_k\in T_S(z)$.
For every $k$ there exist sequences $\tau^k_j\downarrow 0$ and $s^k_j\to 0$ as $j\to\infty$ satisfying $z+\tau^k_j(w+t_kv_k+s^k_j)\in S, \forall \, j$.
For all $k$ sufficiently large we have
 $\bnorm{\frac{w+t_kv_k}{\norm{w+t_kv_k}}-\frac{w}{\norm{w}}}<\frac{\delta}2$
and we can find an index $j(k)$ such that
\[\tau^k_{j(k)}<\frac 1k t_k,\ \norm{s^k_{j(k)}}<\frac 1k t_k,\
\bnorm{\frac{w+t_kv_k+s^k_{j(k)}}{\norm{w+t_kv_k+s^k_{j(k)}}}-\frac{w+t_kv_k}{\norm{w+t_kv_k}}}<\frac \delta2,\
  \tau^k_{j(k)}\norm{w+t_kv_k+s^k_{j(k)}}<\delta.\]
  It follows by Definition \ref{DirectionN} that $z+ \tau^k_{j(k)}(w+t_kv_k+s^k_{j(k)})\in (z+V_{\delta,\delta}(w))\cap S$ and together with $\skalp{z^*,w}=0$ we obtain from \eqref{EqAuxDirProx1} that
  {
 \[\tau^k_{j(k)}t_k\skalp{z^*, v_k+\frac{s^k_{j(k)}}{t_k}}=\skalp{z^*,(z+ \tau^k_{j(k)}(w+t_kv_k+s^k_{j(k)}))-z}\leq \gamma (\tau^k_{j(k)})^2\norm{w+t_kv_k+s^k_{j(k)}}^2.
 \]}
  Dividing this inequality by $\tau^k_{j(k)} t_k$ we conclude that
 \[\skalp{z^*,v}=\lim_{k\to\infty}\skalp{z^*, v_k+\frac{s^k_{j(k)}}{t_k}}\leq \lim_{k\to\infty}\gamma\frac{\tau^k_{j(k)}}{t_k}\norm{w+t_kv_k+s^k_{j(k)}}^2=0.\]
 Hence (\ref{Eq(8)}) holds and therefore the second inclusion in \eqref{EqPropDirProxNormalCone} follows.
\end{proof}

\subsection{Polyhedral sets}

Next, we provide formulas for tangents and normals to polyhedral sets.
A set $D\subset\R^d$ is said to be {\em convex polyhedral} if it is the intersection of finitely many halfspaces, whereas it is said to be {\em polyhedral} whenever it is the union of finitely many convex polyhedral sets.

Polyhedral sets enjoy the following important property, see also \cite[Exercise 6.47]{RoWe98}.

\begin{proposition}[Exactness of tangential approximations, {\cite[Proposition~8.24]{Io17}}]
 \label{lemexact}
 If $D$ is polyhedral and $z \in D$, then
there is an open neighborhood $W$ of $0$ such that
\begin{equation*}
(D-{z})\cap W=T_{D}({z})\cap W,\label{Eqtangent}
\end{equation*}
or, equivalently,
\begin{equation}\label{EqRedPoly}
{D \cap ({z}+W)=({z}+T_D({z}))\cap ({z}+W)}.
\end{equation}
\end{proposition}
\if{
\begin{proof}
{Let  $D:=\displaystyle \cup_{i=1}^s D_i$ where  each $D_i (i=1,\dots, s) $ is  convex polyhedral and $\bar{z} \in
D$.} Then by Rockafellar and Wets \cite[Exercise 6.47]{RoWe98} there is an open neighborhood $V$ of $0$ such that
{$(D_i-\bar{z})\cap V=T_{D_i}(\bar{z})\cap V$}, $i=1,\ldots,s$. Taking the union on both sides of the above equations, since $T_{D}(\bar{z})=\bigcup_{i=1}^s T_{D_i}(\bar{z})$ we obtain (\ref{Eqtangent}).
So (\ref{EqRedPoly}) holds equivalently.
\end{proof}
}\fi

Equation (\ref{eqn7}) in the result below extends \cite[Proposition 13.13]{RoWe98} from convex polyhedral sets to polyhedral sets.
\begin{proposition}\label{Prop2.13}
Let $D$ be a polyhedral set, $z \in D$ and $w\in T_D(z)$. Then
\begin{eqnarray}
 &&T^2_{D}(z;w) = T_{T_D(z)}(w), \label{eqn7}\\
 &&\widehat N_{T_D(z)}(w)= \widehat N_{T^2_D(z; w)}(0)=(T^2_D(z; w))^\circ, \label{eqn8}\\
 &&N_D(z;w)= N_{T_D(z)}(w)=N_{T^2_D(z; w)}(0). \label{eqn9}
 \end{eqnarray}
\end{proposition}

\begin{proof}
Let  $D:=\displaystyle \cup_{i=1}^s D_i$ where  each $D_i (i=1,\dots, s) $ is  convex polyhedral and $z \in
D$. By \eqref{EqRedPoly}, we get
\begin{equation}\label{EqTan}
T_D(z')=T_{z+T_D(z)}(z')=T_{T_D(z)}(z'-z) \quad \forall z'\in D\cap (z+W),
\end{equation}
{where $W$ is  an open neighborhood of 0.}
Consider a tangent direction $w\in T_D(z)$. Then, by \cite[Proposition 13.13]{RoWe98}, we have $T^2_{D_i}(z;w)=T_{T_{D_i}(z)}(w)$ whenever $z\in D_i$ and $w\in T_{D_i}(z)$. Since we have $T^2_{D_i}(z;w)=T_{T_{D_i}(z)}(w)=\emptyset$ for the remaining $i$ by definition, {we obtain (\ref{eqn7}) by}
\begin{equation}\label{EqSecOrdP}
  T^2_{D}(z;w)=\bigcup_{i=1}^s T^2_{D_i}(z;w)=\bigcup_{i=1}^sT_{T_{D_i}(z)}(w)=T_{T_D(z)}(w),
\end{equation}
{where the first and third equation is due to \cite[Proposition 3.37]{BonSh00}.}
Polarization of both side yields $(T^2_{D}(z;w))^\circ=(T_{T_D(z)}(w))^\circ=\widehat N_{T_D(z)}(w)$.
Since $T^2_{D}(z;w) = T_{T_D(z)}(w)$, we have
$$\widehat N_{T^2_D(z; w)}(0)=\widehat N_{T_{T_D(z)}(w)}(0)=\widehat{N}_{T_D(z)}(w),$$
where the last equality follows from the fact that $T_D(z)$ is a closed cone; see e.g., \cite[(3)]{GfrYe20}.
Hence (\ref{eqn8}) holds.
It remains to show (\ref{eqn9}).
For all $z'$ sufficiently close to $z$ we have $\widehat N_D(z')=\widehat N_{T_D(z)}(z'-z)$ by virtue of \eqref{EqTan}. Hence, for every $w\in T_D(z)$ we have
\begin{equation}\label{EqDirLimNormalP}
N_D(z;w)=\{z^*\mv \exists t_k\downarrow 0, w_k\to w, z_k^*\to z^* \ {\rm with} \ z_k^*\in \widehat N_D(z+ t_kw_k)=\widehat N_{T_D(z)}(w_k)\}= N_{T_D(z)}(w).\end{equation}
For $w=0$ we particularly have $N_D(z)=N_{T_D(z)}(0)$.
Since $T_D(z)$ is also polyhedral, the same formula applies
and, taking into account (\ref{EqSecOrdP}), we get
 $N_{T_D(z)}(w)=N_{T_{T_D(z)}(w)}(0)=N_{T^2_D(z; w)}(0)$. {Combining} this equation with (\ref{EqDirLimNormalP}) we obtain (\ref{eqn9}).
\end{proof}

\subsection{Variational geometry of constraint systems under metric subregularity}

Let us mention some basic facts about the tangents and the normals to a set $S$ described by constraints as $S:=g^{-1}(C) = \{x \in \R^n \mv g(x) \in C\}$,
where $g:\R^n\rightarrow \R^d$ and $ C\subset \R^d$.
We will need to use the following concept of directional metric subregularity, which we introduce only in the special case of constraint mappings.

\begin{definition}[Directional metric subregularity, {\cite[Definition 1]{Gfr13a}}] \label{directionMS}
Let $g:\R^n\rightarrow \R^d$, $C\subset \R^d$ and $\bar x \in S:=g^{-1}(C)$.
 We say that the set-valued constraint mapping
 $x \tto g(x)-C$ is metrically subregular at $(\bar{x},0)$ in direction $u \in \mathbb{R}^n$, or that the metric subregularity constraint qualification (MSCQ) holds at $\bar x$ in direction $u$,
 if there exist $\kappa, \delta, \rho>0$ such that
 \begin{equation}\label{eq:MSCQ}
 {\rm dist}(x,S) \leq \kappa \, {\rm dist}(g(x), C), \quad \forall x\in \bar{x}+V_{\delta,\rho}(u).
 \end{equation}
 The infimum of all $\kappa$ for which there are $\delta, \rho>0$ satisfying \eqref{eq:MSCQ} is called the subregularity modulus.
 In the case $u=0$, we simple say that the constraint mapping is metric subregular at $(\bar{x},0)$ or that the MSCQ holds at $\bar x$.
  \end{definition}
  If $g$ is continuously differentiable, then by \cite[Theorem 2.6]{Gfr14a}, a sufficient condition for MSCQ at $\xb$ in direction $u$ is the condition:
  \begin{equation}
  \nabla g(\bar x)^Ty^*=0,\ y^*\in  N_C(g(\bar x);\nabla g(\bar x)u)\ \Longrightarrow\ y^*=0.\label{FOSCMS}
  \end{equation}
  Asking \eqref{FOSCMS} to be satisfied for all nonzero $u \in \R^n$ corresponds to the so-called first-order sufficient condition for metric subregularity (FOSCMS), which implies MSCQ at $\xb$.
  If in addition the graph of the constraint mapping is a closed cone, then the metric subregularity holds locally if and only if it holds globally.
\begin{proposition}[{\cite[Lemma 3]{Gfr19}}] \label{closedcone}
Let $g:\R^n\rightarrow  \R^d$, $ C\subset \R^d$ and assume that $\{(x,y) \in \R^n\times \R^d \mv g(x) - y \in C\}$, the graph of the constraint mapping $x \tto g(x)-C$, is a closed cone.
Then $0 \in S$ and if MSCQ holds at $0$ then there is some $\kappa >0$ such that \eqref{eq:MSCQ} holds for all $x$.
\end{proposition}

In the following proposition, we collect the basic results about tangents, normals, and second-order tangents to set $S$.

  \begin{proposition}\label{Prop2.10}
  Let $g:\R^n\rightarrow \R^d$ be continuously differentiable, let $C\subset \R^d$ be a closed set and consider $\xb \in S:=g^{-1}(C)$.
  Suppose that the constraint mapping $x\tto g(x)-C$ is metrically subregular at $(\bar x,0)$ in direction $\bar u \in \R^n$.
  Then there is a neighborhood $U$ of $\bar u$ such that for every $u \in U$
  one has
  \begin{eqnarray}\label{eqn : T and T_T}
   T_S(\bar x) \cap U = \nabla g(\bar x)^{-1}\big(T_C(g(\bar x))\big) \cap U,
   &&
   T_{T_S(\bar x)}(u) = \nabla g(\bar x)^{-1}\big(T_{T_C(g(\bar x))}(\nabla g(\bar x)u)\big),
   \\ \label{eqn : DirLimNc}
   N_S (\bar x; u) \subset \nabla g(\bar x)^T N_C(g(\bar x);\nabla g(\bar x)u),
   &&
   N_{T_S(\bar x)}(u) \subset \nabla g(\bar x)^T N_{T_C(g(\bar x))}(\nabla g(\bar x)u).
  \end{eqnarray}
  Additionally, if $g$ is twice continuously differentiable and $u \in T_S(\bar x) \cap U$, one has
  \begin{equation*}
  T^2_S(\bar x;u)=\{p \in \R^n \mv \nabla g(\bar x)p+\nabla^2g(\bar x)(u,u)\in T^2_C(g(\bar x);\nabla g(\bar x) u)\}
  \end{equation*}
  and, denoting the subregularity modulus by $\kappa$,
\begin{eqnarray*}
\dist{p,T^2_S(\bar x;u)}\leq \kappa \dist{\nabla g(\bar x)p+\nabla^2g(\bar x)(u,u), T^2_C(g(\bar x);\nabla g(\bar x) u)}\quad  \forall p \in \R^n.
\end{eqnarray*}
Moreover, if there exists a subspace $L\subset \R^d$ such that
  \begin{equation}\label{eq : Lsubspace}
  T_{T_C(g(\xb))}(\nabla g(\xb) u)+L \subset T_{T_C(g(\xb))}(\nabla g(\xb) u) \ \mbox{ and } \ \nabla g(\xb)\R^n +L=\R^d,
  \end{equation}
  then
  \[
    \widehat{N}_{T_S(\bar x)}(u) = \nabla g(\bar x)^T \widehat N_{T_C(g(\bar x))}(\nabla g(\bar x)u).
  \]
  The subregularity assumption as well as the existence of the subspace $L$ satisfying \eqref{eq : Lsubspace} with $u = \bar u$
  are fulfilled, particularly, under the following directional nondegeneracy condition:
    \begin{equation}\label{EqDirNonDegen00}
    \nabla g(\bar x)^T y^*=0,\ y^*\in {\rm span\,} {N_C(g(\bar x);\nabla g(\bar x) \bar u)}\ \Longrightarrow\ y^*=0.
    \end{equation}
\end{proposition}
\begin{proof}
 By Definition \ref{directionMS}, there exists a neighborhood $U$ of $\bar u$ such that $x\tto g(x)-C$ is metrically subregular
 at $\bar x$ in every direction $u \in U$ with the same modulus.
 Thus, the estimate for the directional limiting normal cone comes from \cite[Theorem 3.1]{BeGfrOut19}.
 In \cite[Proposition 5]{GfrYeZh19}, one can find
 all the statements regarding the second-order tangents
 as well as the first formula in \eqref{eqn : T and T_T},
 which means that, locally around any $u \in U$,
 set $T_S(\bar x)$ has the same pre-image structure as $S$.
 By \cite[Lemma 1]{GfrYeZh19}, however, we infer
 that the corresponding constraint mapping $u' \tto \nabla g(\bar x)u' - T_C(g(\bar x))$
 is metrically subregular at $(u,0)$ and so the remaining
 two estimates for $T_{T_S(\bar x)}(u)$ and $N_{T_S(\bar x)}(u)$
 are results of the standard, nondirectional, calculus.
 Moreover, the formula for the regular normal cone
 is from \cite[Theorem 4]{GfrOut16a}.

 Note that the nondegeneracy condition \eqref{EqDirNonDegen00}
 is clearly stronger than FOSCMS, so it obviously implies
 MSCQ at $\xb$ in direction $\bar u$.

 Let us now show that the subspace $\lin{\widehat T_C(g(\bar x);\nabla g(\bar x)\bar u)}$ satisfies \eqref{eq : Lsubspace} with $u = \bar u$.
 The first property follows immediately from Proposition \ref{lem3.1}.
 By the nondegeneracy, we get
 \begin{equation*}
\R^d=\big(\ker \nabla g(\bar x)^T\cap {\rm span\,} {N_C(g(\bar x);\nabla g(\bar x)\bar u)}\big)^\perp= \nabla g(\bar x)\R^n+\lin{\widehat T_C(g(\bar x);\nabla g(\bar x)\bar u)}
 \end{equation*}
 and \eqref{eq : Lsubspace} follows.
\end{proof}

{
For more information about the directional nondegeneracy
\eqref{EqDirNonDegen00} we refer to \cite[Section~2.4]{BeGfrOut20}, where this condition was first introduced
for convex polyhedral set $C$.
Particularly, \cite[Example~2.15]{BeGfrOut20} clarifies
that for a nonzero direction, it is a strictly milder assumption than the standard nondegeneracy \cite[Formula~4.172]{BonSh00},
which corresponds to the case $\bar u = 0$.
We will further utilize directional nondegeneracy in Sections 4 and 5 in the case of polyhedral set $C$,
showing that under \eqref{EqDirNonDegen00} all the four sets in \eqref{eqn : DirLimNc} actually coincide, see Theorem \ref{ThDirNonDegen},
and certain directional multipliers are unique, see Corollary \ref{CorDirNonDegen}.}

\subsection{Generalized support function and second subderivative}

In this final preliminary part, we recall the definitions of the lower generalized support function and state some basic properties.

 \begin{definition}[Lower generalized support function, {\cite{GfrYeZh19}}] \label{def-hat-sigma}
   Given a nonempty closed set $S\subset \R^n$ we define the {\em lower generalized support function to $S$} as the mapping $\hat\sigma_S:\R^n\to \bar \R$ by
   \[
     \hat\sigma_S(z^*):=\liminf_{\tilde z^* \to z^*}\inf_z\{\skalp{\tilde   z^*,z} \mid \tilde z^*\in N_S(z)\}=\liminf_{\tilde z^*\to z^*}\inf_z\{ \skalp{\tilde z^*,z} \mid \tilde z^*\in \widehat N_S(z)\}
     \qquad \forall z^* \in  \mathbb{R}^n.
   \]
   If $S=\emptyset$, then we define $\hat\sigma_S( z^*):=-\infty $ for all $z^*.$
 \end{definition}
 It was shown in \cite{GfrYeZh19} that in general $\hat \sigma_S(z^*)\leq \sigma_S(z^*)$ for all $z^*$ and the equality holds when $S$ is convex.
 {If $S=z+K$ is a translation of a cone $K$, we get the following formula,
 which will come in handy is Section 5.}

  \begin{proposition}\label{PropHatSigma}
    For every nonempty closed cone $K\subset \mathbb{R}^n$ (not necessarily convex) and every $z\in\R^n$ we have
    \[\hat\sigma_{z+K}(z^*)=\begin{cases}\skalp{z^*,z}&\mbox{if $z^*\in N_K(0)$,}\\
    \infty&\mbox{otherwise.}\end{cases}\]
    Particularly, ${\rm dom\,}\hat\sigma_{z+K} 
    = N_K(0)$.
 \end{proposition}

 \begin{proof}
 First, note that since $K$ is assumed to be a cone, for all $q\in K$ we have
     \begin{align}\label{EqCone1}\widehat N_{z+K}(z+q)=\widehat N_K(q)= \widehat N_K(\alpha q)=\widehat N_{z+K}(z+\alpha q)\ \  \forall \alpha>0,\\
     \label{EqCone2}\skalp{z^*,q}=0\ \  \forall z^*\in \widehat N_K(q).
     \end{align}
     We shall show that $\hat\sigma_{z+K}(z^*)<\infty$ if and only if $z^*\in N_K(0)$ and in this case we have $\hat\sigma_{z+K}(z^*)=\skalp{z^*,z}$.
      If $\hat\sigma_{z+K}(z^*)<\infty$, then there exist  sequences $z^*_k\to z^*$ and $q_k\in K$ with $z^*_k\in\widehat N_{z+K}(z+q_k)$ for all $k$ such that
    $\hat \sigma_{z+K}(z^*)=\displaystyle \lim_{k\to\infty}\skalp{{z^*_k},z+q_k}.$
    By \eqref{EqCone1} we have $z^*_k\in\widehat N_{z+K}(z+q_k)=\widehat N_{K}(q_k)$ and hence
    $\skalp{z^*_k, q_k}=0$ by
    \eqref{EqCone2}. It follows that
  \[\hat \sigma_{z+K}(z^*)= \lim_{k\to\infty}\skalp{{z^*_k},z}=\skalp{z^*,z}.\]
     Moreover, by \eqref{EqCone1} we have $z^*_k\in \widehat N_K(\alpha_kq_k)$ for $\alpha_k := 1/(k(\Vert q_k\Vert+1))$.
     Taking the limit as $k$ goes to $\infty$ we obtain $z^*\in N_K(0)$.
      Conversely, let $z^*\in N_K(0)$ and consider sequences $q_k\in K$ and $z^*_k\in\widehat N_K(q_k)$ such that $q_k\to 0$ and $z^*_k\to z^*$. Then
      $z^*_k \in \widehat N_K(q_k)= \widehat N_{z+K}(z+q_k)$ by \eqref{EqCone1} and $\skalp{z^*_k,q_k}=0$ by \eqref{EqCone2}. Hence by Definition \ref{def-hat-sigma} we obtain that
     \[\hat\sigma_{z+K}(z^*)\leq \liminf_{k\to\infty}\skalp{z^*_k,z+q_k}=\skalp{z^*,z} <\infty.\]
 \end{proof}
\begin{definition}[Second subderivative, {\cite[Defintion 13.3]{RoWe98}}]
Let $\varphi: \R^n\to \bar \R$, $\varphi(z)$ be finite and $z^*\in  \mathbb{R}^n$. The second subderivative of $\varphi$ at $z$ for $z^*$  is a function defined by
 \begin{eqnarray*}
 {\rm d}^2\varphi(z;z^*)(w):= \liminf\limits_{{t\downarrow 0} \atop {w'\to w}} \frac{\varphi(z+tw')-\varphi(z)-t\langle z^*, w'\rangle}{\frac{1}{2}t^2}
 \qquad \forall w \in  \mathbb{R}^n.
\end{eqnarray*}
 \end{definition}
  According to \cite[Example 13.8]{RoWe98}, if $\varphi$ is twice differentiable at $z$ and $z^*=\nabla \varphi(z)$, one has
 \[{\rm d}^2\varphi(z;z^*)(w)=w^T \nabla \varphi^2(z)w.\]
By definition, the second subderivative of the indicator function $\delta_S$ of a set
$S$ at $z\in S$ for $z^*$ is
\begin{eqnarray}
{\rm d}^2\delta_{S}(z;z^*)(w)
 =\liminf\limits_{{t\downarrow 0} \atop w'\to w}\frac{\delta_{S}(z+tw')-\delta_{S}(z)-t\langle z^*, w' \rangle }{\frac{1}{2}t^2}
=\liminf\limits_{{t\downarrow 0, w'\to w} \atop { z+tw' \in S }}\frac{-2\langle z^*, w' \rangle }{t}. \label{indicatorf}
 \end{eqnarray}
 The second subderivative of the indicator function  is extended-real-valued and, by definition,  a function of the direction $w$. However, when dealing with second-order optimality conditions, it also makes sense to consider its dependence on $z^*$. In the following proposition we will investigate the set of all $(z^*,w)$ such that $ {\rm d}^2\delta_{S}(z;z^*)(w)$ is finite and the relationship between
${\rm d}^2\delta_{S}(z;z^*)(w)$ and the support function of the second-order tangent cone $\sigma_{T^2_S(z;w)}(z^*)$ as well as with the lower generalized support function $\hat{\sigma}_{T^2_S(z;w)}(z^*)$. It turns out that the directional proximal normal cones are useful in characterizing the points where ${\rm d}^2\delta_{S}(z;z^*)(w)$ is finite.
 \begin{proposition}\label{LemBasicSecOrdSubDer}
 Consider a closed set $S\subset\R^n$, $z \in S$ and a pair $(w,z^*)\in\R^n\times\R^n$.
 The following statements hold:
\begin{enumerate}
  \item[\rm (i)] If $w\not\in T_S(z)$ or $\skalp{z^*,w}<0$, then ${\rm d}^2\delta_S(z;z^*)(w)=\infty$.
  \item[\rm (ii)] For $w\in T_S(z)$, we have ${\rm d}^2\delta_S(z;z^*)(w)>-\infty$ if and only if $z^*\in \Np_S(z;w)$.
  \item[\rm (iii)] If ${\rm d}^2\delta_S(z;z^*)(w)$ is finite, then $z^*\in \widehat N_S^p(z;w)$.
  \item[\rm (iv)] We have
  \[
  {\rm d}^2\delta_S(z;z^*)(w)
  \leq
  -\sigma_{T^2_S(z;w)}(z^*)
  \leq
  - \hat{\sigma}_{T^2_S(z;w)}(z^*)
  \]
  if and only if $w\in T_S(z)$ and $\langle z^*, w\rangle \geq 0$ or $T^2_S(z;w)=\emptyset$.
\end{enumerate}
\end{proposition}

\begin{proof}
(i) The statement follows from the definition of tangent cone and (\ref{indicatorf}).

 (ii) In order to show the if-part of the statement, let $w\in T_S(z)$ and consider $z^*\in \Np_S(z;w)$. Then we can find some 
  $\delta>0, \gamma>0$ such that
  \begin{equation}\label{Eq(11)}\skalp{z^*,z'-z}
{\leq}\gamma\norm{z'-z}^2\quad  \forall z'\in(z+V_{\delta,\delta}(w))\cap S.\end{equation}
   On the other hand, by \eqref{indicatorf}, we can also find sequences $t_k\downarrow 0, w_k\to w$ such that $z+t_kw_k\in S$ and
${\rm d}^2\delta_S(z;z^*)(w)
=\lim_{k\to\infty}\frac{-2\skalp{z^*,w_k}}{ t_k}.$
  Since $w_k\rightarrow w$, for all $k$ sufficiently large we have $t_kw_k\in V_{\delta,\delta}(w)$. By (\ref{Eq(11)}) we have $\skalp{z^*,t_kw_k}
{\leq}\gamma t_k^2\norm{w_k}^2$, from which  we obtain the desired inequality
  \[{\rm d}^2\delta_S(z;z^*)(w)=\lim_{k\to\infty}\frac{-2\skalp{z^*,w_k}}{ t_k}
{\geq} \lim_{k\to\infty}-2\gamma\norm{w_k}^2=-2\gamma\norm{w}^2>-\infty.\]
  In order to show the only if-part, assume on the contrary that $w\in T_S(z)$ and $z^*\not\in \Np_S(z;w)$.
Then there are sequences $t_k\downarrow 0$ and $w_k\to w$ such that $z_k:=z+t_kw_k\in S$ and \[\limsup_{k\to\infty}\frac{\skalp{z^*,z_k-z}}{\norm{z_k-z}^2}=\limsup_{k\to\infty}\frac{\skalp{z^*,w_k}}{t_k\norm{w_k}^2}=\infty.\]
  If $w\not=0$ the contradiction
  \[\infty =\limsup_{k\to\infty}\frac{2\skalp{z^*,w_k}}{t_k}
  \leq \limsup\limits_{{t\downarrow 0, w'\to w} \atop { z+tw'\in {S}}}\frac{2\langle z^*, w' \rangle }{t}=-{\rm d}^2\delta_S(z;z^*)(w)\]
  follows. In case when $w=0$, after possibly passing to a subsequence, we can assume that $\norm{w_k}<\frac 1k$ and $\skalp{z^*,w_k}/(t_k\norm{w_k}^2)>k^3$ holds for all $k$. Defining
  $\tilde t_k:=t_k\norm{w_k}k$, $\tilde w_k:=w_k/(k\norm{w_k})$, we have $z_k=z+\tilde t_k\tilde w_k\in S$, $\tilde t_k\downarrow 0$ and $\tilde w_k\to 0$ and therefore we  obtain once more the contradiction
  \[{\rm d}^2\delta_S(z;z^*)(0)=
  \liminf\limits_{{t\downarrow 0, w'\to 0} \atop { z+tw'\in S }}\frac{-2\langle z^*, w' \rangle }{t}\leq
{\liminf_{k\to\infty}\frac{-2\skalp{z^*,\tilde w_k}}{\tilde t_k}=}\liminf_{k\to\infty}\frac{-2\skalp{z^*,w_k}}{k^2t_k\norm{w_k}^2}\leq \liminf_{k\to\infty}-k=-\infty.\]
  The above arguments show that $z^*\in \Np_S(z;w)$.

  (iii) If ${\rm d}^2\delta_S(z;z^*)(w)$ is finite then by the statement (i), we must have $w\in T_S(z)$ and $\skalp{z^*,w}\geq 0$.
It then follows by the statement (ii) that $z^*\in \Np_S(z;w)$.
Since ${\rm d}^2\delta_S(z;z^*)(w)=-\infty$ as $\langle z^*, w\rangle>0$ by definition, we obtain $\langle z^*, w\rangle=0$.
  Thus $z^*\in \widehat N_S^p(z;w)$ and the third assertion is shown.

  (iv) According to \cite[{Proposition 6}]{GfrYeZh19}, we know $-\sigma_{T^2_S(z;w)}(z^*)\leq -\hat{\sigma}_{T^2_S(z;w)}(z^*)$ for all $z^*$.
Hence it remains to show
\begin{equation}\label{inequ-d2}
{\rm d}^2\delta_S(z;z^*)(w)\leq -\sigma_{T^2_S(z;w)}(z^*)
\end{equation}
  if and only if $w\in T_S(z)$ and $\langle z^*, w\rangle \geq 0$ or $T^2_S(z;w)=\emptyset.$
{For necessity, if $T^2_S(z;w)=\emptyset$, then $-\sigma_{T^2_S(z;w)}(z^*)=\infty$ and hence
(\ref{inequ-d2}) holds.
Let $w\in T_S(z)$.
If $\langle z^*, w\rangle > 0$, then (\ref{inequ-d2}) follows from
${\rm d}^2\delta_S(z;z^*)(w)=-\infty$, while if
$\langle z^*, w\rangle=0$, it holds by \cite[Proposition 3.2]{MoMoSa20}.
We prove the sufficiency by contradiction.
Suppose that (\ref{inequ-d2}) holds but $T^2_S(z;w)\neq \emptyset$ and either
$w\notin T_S(z)$ or $\langle z^*, w\rangle<0$. In this case we must have ${\rm d}^2\delta_S(z;z^*)(w)=\infty$ by statement (i).
On the other hand, since  $T^2_S(z;w)\neq \emptyset$, then $-\sigma_{T^2_S(z;w)}(z^*)<+\infty$.
Hence ${\rm d}^2\delta_S(z;z^*)(w)>-\sigma_{T^2_S(z;w)}(z^*)$, contradicting (\ref{inequ-d2}).}
\end{proof}

\section{Second-order optimality conditions for (GP)}
Recall the general problem
\begin{flalign*}
\begin{split}
\mbox{(GP)} \hspace{49mm} \min & \ \  f(x)\ \ \  {\rm s.t. } \ \ g(x)\in C
\end{split}&
\end{flalign*}
from the introduction. At a feasible point $\xb$ of (GP), the critical cone is defined as
    \[ \mathcal{C}(\xb):=\{u \in \R^n \mv \nabla g(\xb)u\in T_C(g(\xb)), \nabla f(\xb) u \leq 0\}\]
     and the generalized Lagrangian $L^\alpha:\R^n\times\R^m\to \R$ with $\alpha\geq 0$ is given as
     \[
	L^\alpha(x,\lambda):=\alpha f(x)+ g(x)^T \lambda,
     \]
	where for $\alpha = 1$ we get the standard Lagrangian $L:=L^1$.
    To study optimality conditions for (GP), we define various multiplier sets as follows, where $u\in \mathcal{C}(\xb)$ denotes a critical direction:
\begin{eqnarray*}
\Lambda(\xb; u)&:=&\{\lambda \in N_C(g(\xb);\nabla g(\xb) u) \mv \nabla_x L(\xb,\lambda)=0\} \mbox{ (directional M-mulipliers)},\\
\Lambda^s(\xb;  u)&:=&\{\lambda \in \widehat N_{T_C(g(\xb))}(\nabla g(\xb) u) \mv \nabla_x L(\xb,\lambda)=0 \}\mbox{ (directional S-mulipliers)},\\
\Lambda^p(\xb; u)&:=&\{\lambda \in \widehat N_C^p(g(\xb);\nabla g(\xb) u) \mv \nabla_x L(\xb,\lambda)=0\}\mbox{ (directional proximal mulipliers)}.
\end{eqnarray*}
For $u=0$ we speak of just M-, S-, and proximal multipliers, which we denote by $\Lambda(\xb):=\Lambda(\xb;0)$, $\Lambda^s(\xb):=\Lambda^s(\xb;0)$, and $\Lambda^p(\xb):=\Lambda^s(\xb;0)$, respectively.
For every $u\in \mathcal{C}(\xb)$ and every $\lambda\in \Lambda^p(\xb)\subset \widehat N^p_C(g(\xb))\subset \widehat N_C(g(\xb))=(T_C(g(\xb)))^\circ$,  we have
\[0\leq -\nabla f(\xb)u=\lambda^T\nabla g(\xb)u\leq 0\]
implying $\lambda^T\nabla g(\xb) u =0$. Hence, by virtue of Proposition \ref{PropDirProxNormalCone}, the following relations hold
$$ \Lambda^p(\xb) \subset  \Lambda^p(\xb; u) \subset  \Lambda^s(\xb; u) \subset  \Lambda(\xb; u)$$
and the inclusions become equalities provided {$C$ is convex by (\ref{equality-convex})}.
    In general we only have the inclusion $\Lambda^p(\xb) \subset \Lambda^s(\xb)$ but they are equal for many {nonconvex and nonpolyhedral sets important  in applications}; e.g., {the second-order cone complementarity set \cite{YZ17} and the semidefinite complementarity cone \cite{DSY14}.}

Recall first the following second-order necessary optimality condition for (GP).
\begin{theorem}[{\cite[Theorem 2 and Corollary 5]{GfrYeZh19}}]\label{Thm3.1}
 Let $\bar x$ be a local optimal solution for  problem (GP).
Then for every critical direction $u\in \mathcal{C}(\xb)$ the following necessary optimality conditions hold.
 \begin{itemize}
 \item[\rm (i)] Suppose the constraint mapping $x\tto g(x)-C$ is metrically subregular at $(\xb,0)$ in direction $u$. Then there exists a directional M-multiplier $\lambda \in \Lambda(\xb; u)$ such that
\begin{equation} \nabla^2_{xx}L(\bar{x},\lambda)(u,u) -\hat{\sigma}_{T_C^2(g(\xb); \nabla g(\xb)u)}(\lambda)\geq 0.\label{lambda}
\end{equation}
\item[\rm (ii)] Suppose that the directional nondegeneracy condition
$$   \nabla g(\bar x)^Ty^*=0,\ y^* \in {\rm span\,} {N_C(g(\bar x);\nabla g(\bar x)u)}\  \Longrightarrow \ y^*=0 $$
 holds. Then $\Lambda^s(\xb; u)=\Lambda(\xb; u)=\{\lambda_0\}$ is a singleton and the second-order condition
 \begin{equation*}
 \nabla^2_{xx}L(\bar{x},\lambda_0)(u,u)-\sigma_{T_C^2(g(\xb); \nabla g(\xb)u)}(\lambda_0)\geq 0\label{lambda0}
\end{equation*}
holds.
 \end{itemize}
 \end{theorem}

We now derive second-order sufficient optimality conditions for (GP).
We state our result in terms of the following notion introduced by Penot \cite{Penot}.
 \begin{definition}[Essential local minimizer of second order]
 A point $\bar{x}$ is said to {be} an essential local minimizer of second order for problem (GP) if $\bar x$ is feasible and there exist $\varepsilon>0$ and $\delta>0$ such that
 \[
 \max\{f(x)-f(\bar{x}), {\rm dist}(g(x),C) \}\geq \varepsilon \|x-\bar{x}\|^2, \ \ \forall x\in \mathbb{B}(\bar{x},\delta).
 \]
 \end{definition}
\begin{theorem}\label{Theorem5.9}
 Let $\bar x$ be a feasible point of problem (GP).
 Suppose that for every $u\in \mathcal{C}(\bar x ) \backslash \{0\}$ there is $\alpha\geq 0$ and $\lambda\in \R^m$ such that
 \begin{equation}\label{Alpha}
 \nabla_x L^\alpha (\bar x,\lambda)=0\end{equation}
  and
  \begin{equation}\label{alpha-sufficient}
   \nabla^2_{xx}L^\alpha(\bar{x},\lambda)(u,u)+{\rm d}^2\delta_{C}(g(\bar{x});\lambda)(\nabla g(\bar{x})u)>0.
  \end{equation}
  Then $\bar{x}$ is an essential local minimizer of second order.
 \end{theorem}
 \begin{proof}
 By contradiction, if $\bar{x}$ is not an essential local minimizer of second order, then there exists a sequence $x_k$ converging to $\xb$ such that
\begin{eqnarray}
&&  f(x_k)-f(\bar{x})\leq o(\|x_k-\bar{x}\|^2), \label{cqg}\\
&&  {\rm dist}(g(x_k),C)\leq o(\|x_k-\bar{x}\|^2).\label{distance}
 \end{eqnarray}
 Let $t_k:=\|x_k-\bar{x}\|$ and $u_k:=(x_k-\bar{x})/t_k$.
 We assume without loss of generality that $u_k$ is converging to $u$.
 From equation \eqref{cqg} it readily follows that $\nabla f(\xb)u\leq 0$ and
  \begin{equation}\label{f-inequality}
 \liminf\limits_{k}-\frac{f(x_k)-f(\bar{x})}{\frac{1}{2}t_k^2}\geq 0.
 \end{equation}
 By (\ref{distance}), there exists {$r_k\in \R^d$} such that $\|r_k\| \to 0$ and
 $g(x_k)+t^2_k r_k\in C$.
 By Taylor's expansion, since $x_k=\bar x+t_ku_k$, we have
 \[
v_k:= \frac{g(x_k)+ t^2_kr_k - g(\xb)}{t_k} = \frac{\nabla g(\bar{x})t_ku_k +o(t_k) +  t^2_kr_k}{t_k} = \nabla g(\bar{x})u_k +\frac{o(t_k)}{t_k}{+t_kr_k} \ \to \ \nabla g(\bar{x})u.\]
Moreover, $g(\bar{x})+t_kv_k=g(x_k)+ t^2_kr_k\in C$ and so $\nabla g(\bar{x})u\in T_C(g(\xb))$ follows.
Thus $u\in \mathcal{C}(\xb)\setminus\{0\}$ and the assumption of the theorem yields the existence of $\alpha\geq 0$ and $\lambda\in\R^m$ satisfying \eqref{Alpha} and \eqref{alpha-sufficient}.
Using \eqref{indicatorf}, \eqref{f-inequality}, $\|r_k\| \to 0$, and \eqref{Alpha}, however, we obtain
 \begin{eqnarray*} 
	{\rm d}^2\delta_{C}(g(\bar{x});\lambda)(\nabla g(\bar{x})u)
	&=&
	\liminf\limits_{{t\downarrow 0, v'\to \nabla g(\bar{x})u} \atop { g(\bar{x})+tv'\in C }}\frac{-\langle \lambda, v' \rangle }{\frac{1}{2}t}
	 \leq
	\liminf\limits_{k}\frac{-t_k\langle \lambda, v_k \rangle }{\frac{1}{2}t_k^2} \\
	&=&
	\liminf\limits_{k}\frac{-\langle \lambda, g(x_k)+ t^2_kr_k-g(\bar{x})  \rangle }{\frac{1}{2}t_k^2} \\
	&\leq&
	\liminf\limits_{k}-\alpha\frac{f(x_k)-f(\bar{x})}{\frac{1}{2}t^2_k}+ \liminf\limits_{k}-\frac{\langle \lambda, g(x_k)-g(\bar{x}) \rangle }{\frac{1}{2}t^2_k}\\
	&\leq&
	\liminf\limits_{k}-\frac{L^\alpha (x_k,\lambda)-L^\alpha(\bar{x},\lambda)}{\frac{1}{2}t_k^2}\\
	&=&
	\liminf\limits_{k}-\frac{\nabla_x L^\alpha(\bar{x},\lambda)(t_ku_k)+\frac 12 \nabla^2_{xx} L^\alpha(\bar{x},\lambda)(t_ku_k,t_ku_k)+\oo(t_k^2)}{\frac{1}{2}t_k^2}\\
	&=&
	\liminf\limits_{k}-\nabla^2_{xx} L^\alpha(\bar{x},\lambda)(u_k,u_k)=-\nabla^2_{xx}L^\alpha(\bar{x},\lambda)(u,u),
 \end{eqnarray*}
which contradicts \eqref{alpha-sufficient}. This completes the proof.
\end{proof}
Additional requirements on $\lambda$ are hidden in conditions \eqref{Alpha} and \eqref{alpha-sufficient}.
\begin{proposition}
Let $\bar x$ be a feasible point of problem (GP), let $u\in \mathcal{C}(\xb)$ be a critical direction and let $\alpha\geq 0$, $\lambda\in\R^m$ satisfy conditions \eqref{Alpha} and \eqref{alpha-sufficient}.
Then $\alpha$ and $\lambda$ are not both zero and $\lambda\in \widehat{N}_C^p(g(\bar x); \nabla g(\bar x) u)$.
Particularly, if $\alpha \neq 0$, then $\tilde\lambda:=\lambda/\alpha \in\Lambda^p(\xb; u)$ and conditions \eqref{Alpha}-\eqref{alpha-sufficient} hold with $\tilde\alpha :=1$ and $\tilde\lambda$.
\end{proposition}
\begin{proof}
Note that $\alpha$ and $\lambda$ cannot be simultaneously zero because otherwise
  \[\nabla^2_{xx}L^\alpha(\bar{x},\lambda)(u, u)= {\rm d}^2\delta_{C}(g(\bar{x}); \lambda)(\nabla g(\bar{x})u)=0,\]
  contradicting \eqref{alpha-sufficient}. Since
   ${\rm d}^2\delta_{C}(g(\bar{x});\lambda)(\nabla g(\bar{x})u)>-\infty,$
     we conclude $\lambda\in \Np_C(g(\bar x); \nabla g(\bar x)u)$ by Proposition \ref{LemBasicSecOrdSubDer}(ii) and $\langle \lambda, \nabla g(\bar x)u \rangle\leq 0$ by (\ref{prenormal}).
Meanwhile, $\lambda$ satisfies $\nabla_xL^\alpha(\bar x,\lambda)=0$, i.e., $\alpha\nabla f(\bar x)+\nabla g(\bar x)^T{\lambda}=0$, implying
$\langle \lambda, \nabla g(\bar x)u\rangle = - \alpha\nabla f(\bar x)u \geq 0$ due to $u\in \mathcal{C}(\bar x)$.
Thus $\langle \lambda, \nabla g(\bar x)u\rangle=0$ and we get
$$\lambda \in \Np_C(g(\bar x); \nabla g(\bar x)u)\cap \{\nabla g(\bar x)u\}^\perp=\widehat{N}_C^p(g(\bar x); \nabla g(\bar x)u).$$
Since $\widehat{N}_C^p(g(\bar x); \nabla g(\bar x)u)$ is a cone, it also contains $\lambda/\alpha$ if $\alpha \neq 0$.
Thus, dividing \eqref{Alpha}-\eqref{alpha-sufficient} by $\alpha$ yields the last claim, taking into account \eqref{indicatorf}.
\end{proof}

Let us now compare the concept of essential local minimizers with the more common notion that the quadratic growth condition for (GP)  holds at $\xb$, i.e.,  there exist $\varepsilon>0$ and $\delta>0$ such that  \begin{equation} f(x)\geq  f(\bar x)+\varepsilon\|x-\bar x\|^2 \qquad \forall x\in  \mathbb{B}(\bar x,  \delta) \ \  s.t.  \ \ g(x)\in C.\label{grownew}
   \end{equation}
\begin{lemma}
  Consider the following statements:
  \begin{enumerate}
    \item[\rm (i)]$\xb$ is an essential local minimizer of second-order.
    \item[\rm (ii)]The quadratic growth condition holds at $\xb$.
  \end{enumerate}
  Then the implication (i) $\Rightarrow$ (ii) always hold. Conversely, if the constraint mapping $x\tto g(x)-C$ is metrically subregular at $(\xb,0)$ in every critical direction $u\in\mathcal{C}(\xb)\setminus\{0\}$ then the reverse implication (ii) $\Rightarrow$ (i) is also valid.
\end{lemma}
\begin{proof}
The validity of the implication (i) $\Rightarrow$ (ii) follows immediately from the definitions. We show the second assertion by contraposition. Assume that the quadratic growth condition and the stated constraint qualification hold and assume on the contrary that there is a sequence $x_k\to\xb$ with $\max\{f(x_k)-f(\bar{x}), {\rm dist}(g(x_k),C) \}/\norm{x_k-\xb}^2\to 0$. By passing to a subsequence we may assume that $(x_k-\xb)/\norm{x_k-\xb}$ converges to some $u$ and the arguments already employed in the proof of  Theorem \ref{Theorem5.9} show that $u\in {\mathcal C}(\xb) \setminus\{0\}$. By the assumed directional metric subregularity, there is some $\kappa>0$ such that for all $k$ sufficiently large we can find some $\tilde x_k$ with $g(\tilde x_k)\in C$ and $\norm{\tilde x_k-x_k}\leq \kappa\dist{g(x_k),C}=\oo(\norm{x_k-\xb}^2)$. Since $f$ is Lipschitz continuous in a neighborhood of $\xb$ with some constant $l$, we obtain from \eqref{grownew}
the contradiction
\begin{align*}
0&<\varepsilon\leq\liminf_{k\to\infty}\frac{f(\tilde x_k)-f(\xb)}{\norm{\tilde x_k-\xb}^2}\leq \liminf_{k\to\infty}\frac{f( x_k)-f(\xb)+l\norm{\tilde x_k-x_k}}{(\norm{ x_k-\xb}-\norm{\tilde x_k-x_k})^2}\\
&=\liminf_{k\to\infty}\frac{f( x_k)-f(\xb)+\oo(\norm{x_k-\xb}^2)}{\norm{ x_k-\xb}^2-\oo(\norm{x_k-\xb}^2)}
=\liminf_{k\to\infty}\frac{f( x_k)-f(\xb)}{\norm{ x_k-\xb}^2}\leq 0.
\end{align*}
\end{proof}
Note that Theorem \ref{Theorem5.9} improves Mohammadi et al. \cite[Proposition 7.3]{MoMoSa20} in that the set $C$ is not required to be convex and parabolically derivable, $\alpha$ can be zero, we can choose different multipliers for different critical directions  and the concept of local minimizer is stronger. For the sake of completeness we also state the following corollary.
\begin{corollary} Let $\bar x$ be a feasible point of problem (GP).
 Suppose that  for every $u\in \mathcal{C}(\bar x ) \backslash \{0\}$  there is a directional proximal multiplier $\lambda\in \Lambda^p(\xb; u)$ such that 
{\[\nabla^2_{xx}L(\xb,\lambda){(u,u)}+{\rm d}^2\delta_{C}(g(\bar{x});\lambda)(\nabla g(\bar{x})u)>0.\]}
  Then the quadratic growth condition (\ref{grownew}) holds for problem (GP).
\end{corollary}
\section{First-order variational analysis of disjunctive systems}

In this section, we begin with first-order variational analysis of the disjunctive system $\Gamma:=\{x \in \R^n \mv {G}(x)\in {D}\}$,
where ${G}:\R^n \to \R^d$ is continuously differentiable and ${D} \subset \R^d$ is polyhedral.
Note that many first-order results are valid for any closed set $D$ and were already stated in Proposition \ref{Prop2.10}.
In the following theorem, we show that  since $D$ is polyhedral,
some results from Proposition \ref{Prop2.10} can be improved.
Namely, under the directional nondegeneracy \eqref{EqDirNonDegen00}
the inclusions in \eqref{eqn : DirLimNc} become equalities,
making all four sets equal.
Note also that \cite[Theorem 6.14]{RoWe98} cannot be applied directly since ${D}$ is not regular in the sense of Clarke (cf. \cite[Definition 6.4]{RoWe98}).

\begin{theorem}\label{ThDirNonDegen}
   Consider a feasible point $\bar x\in \Gamma$ and a direction $u \in \R^n$
    and assume that {the directional nondegeneracy condition}
    \begin{equation}\label{EqDirNonDegen1}
    \nabla {G}(\bar x)^Ty^*=0,\ y^*\in {\rm span\,} {N_{D}({G}(\bar x);\nabla {G}(\bar x)u)}\ \Longrightarrow\ y^*=0
    \end{equation}
    is fulfilled. Then
    \begin{equation}\label{Eq(13-14)}
     N_\Gamma(\bar x;u)
     =
     \nabla {G}(\bar x)^TN_{D}({G}(\bar x);\nabla {G}(\bar x)u)
     =
     \nabla {G}(\bar x)^TN_{T_{D}({G}(\bar x))}(\nabla {G}(\bar x)u)
     =
     N_{T_\Gamma(\bar x)}(u).
    \end{equation}
    \if{
    \begin{align}
    \widehat N_{T_\Gamma(\bar x)}(u)&= \nabla {G}(\bar x)^T\widehat N_{T_{D}({G}(\bar x))}(\nabla {G}(\bar x)u),\label{Eq(13)} \\
    N_\Gamma(\bar x;u)&=\nabla {G}(\bar x)^TN_{D}({G}(\bar x);\nabla {G}(\bar x)u).\label{Eq(14)}
    \end{align}
    }\fi
\end{theorem}
\begin{proof}
{First note that condition (\ref{EqDirNonDegen1}) implies FOSCMS (\ref{FOSCMS}) and hence $x \tto G(x)-D$ is metrically subregular at $(\bar{x},0)$ in direction $u$.}  From \cite[Lemma 3]{GfrYeZh19}, Proposition \ref{Prop2.10}, and  Proposition \ref{Prop2.13}, respectively, we know
\[
N_{T_\Gamma(\bar x)}(u)
\subset
N_\Gamma(\bar x;u)
\subset
\nabla {G}(\bar x)^TN_{D}({G}(\bar x);\nabla {G}(\bar x)u)
=
\nabla {G}(\bar x)^TN_{T_{D}({G}(\bar x))}(\nabla {G}(\bar x)u).
\]
Thus, it remains to show
$\nabla {G}(\bar x)^TN_{T_{D}({G}(\bar x))}(\nabla {G}(\bar x)u)
\subset
N_{T_\Gamma(\bar x)}(u)$.
Consider $y^*\in N_{T_{D}({G}(\bar x))}(\nabla {G}(\bar x)u)$
together with sequences
$v_k \in T_{D}({G}(\bar x))$ and $y^*_k \in \widehat N_{T_{D}({G}(\bar x))}(v_k)$
with $v_k \to \nabla {G}(\bar x)u$ and $y^*_k \to y^*$.
The nondegeneracy condition and Proposition \ref{lem3.1} yield
\begin{equation*}
  \R^d
  =
  \big(\ker \nabla {G}(\bar x)^T\cap {\rm span\,}{ N_{D}({G}(\bar x); \nabla {G}(\bar x)u)
}\big)^\perp
=
\nabla {G}(\bar x)\R^n+ \lin{T_{T_{D}({G}(\bar x))}(\nabla {G}(\bar x)u)}.
 \end{equation*}
{Consider the $d\times(n+d)$ matrix $A:=\big(\nabla G(\xb)\,\vdots\,P\big)$, where $P$ is the symmetric $d\times d$ matrix representing the orthogonal projection onto $\lin{T_{T_{D}({G}(\bar x))}(\nabla {G}(\bar x)u)}$. By the equation above, $A$ has full row rank $d$ and therefore the pseudo-inverse of $A$ is given by $A^\dag:=A^T(AA^T)^{-1}$. It follows that 
\[\myvec{d_k\\r_k}:=A^\dag(v_k - \nabla {G}(\bar x)u)\to\myvec{0\\0}\]
and $v_k- \nabla {G}(\bar x)u = \nabla {G}(\bar x)d_k + Pr_k$ for all $k$. Since the pseudo-inverse computes the minimum norm solution to a linear system and $\norm{Pr_k}\leq\norm{r_k}$, we conclude $r_k=Pr_k\in \lin{T_{T_{D}({G}(\bar x))}(\nabla {G}(\bar x)u)}$. Hence, $u_k:=u +d_k \to u$ and $\nabla {G}(\bar x)u_k = v_k -r_k\to \nabla {G}(\bar x)u$ follows.
}

\if{ {\color{red} Matus: If we do it this way, we may erase the above equalities, right? But maybe we can rather refer to the Hofmann lemma (or explain the linear algebra argument) and keep the original proof.}
 {\color{blue}Notice that the assumption (\ref{EqDirNonDegen1}) ensures the metric regularity of the mapping $\widetilde{M}(\xi):=\nabla G(\bar x)\xi-({\rm span}N_D(G(\bar x); \nabla G(\bar x)u))^\perp$ at $(u,\nabla G(\bar x)u)$ by Mordukhovich criterion (cf. e.g., \cite[Example 9.44]{RoWe98}). Hence there exists $\beta>0$ such that
 \begin{equation}\label{add-beta}
 {\rm dist}(u, \widetilde{M}^{-1}(v_k))\leq \beta {\rm dist} (v_k, \widetilde{M}(u))\leq \beta \|v_k-\nabla G(\bar x)u\|.
 \end{equation}
 Let $u_k$ be the projection of $u$ to $\widetilde{M}^{-1}(v_k)$. Then there exists $r_k\in ({\rm span}N_D(G(\bar x); \nabla G(\bar x)u))^\perp$ satisfying
 $v_k=\nabla G(\bar x)u_k-r_k$. According to Proposition \ref{lem3.1}, $r_k\in \lin{T_{T_D(G(\bar x))}(\nabla G(\bar x)u)}$. So (\ref{add-beta}) takes the form
 $\|u-u_k\|\leq \beta \|v_k-\nabla G(\bar x)u\|$, implying $u_k\to u$.}
}\fi

We obtain
 \begin{eqnarray}
 y^*_k & \in & \widehat N_{T_{D}({G}(\bar x))}(v_k) =\widehat N_{T_{T_{D}({G}(\bar x))}(\nabla {G}(\bar x)u)}(v_k-\nabla {G}(\bar x)u) \nonumber\\
 &=& \widehat N_{T_{T_{D}({G}(\bar x))}(\nabla {G}(\bar x)u)}(v_k-\nabla {G}(\bar x)u {- r_k})=\widehat N_{T_{D}({G}(\bar x))}(v_k {-r_k}) = \widehat N_{T_{D}({G}(\bar x))}(\nabla {G}(\bar x)u_k), \label{add-beta-1}
\end{eqnarray}
 where {the first and the third equalities follow from \eqref{EqRedPoly} and the second equality comes from (\ref{EqConsequLin}).}
Moreover, since $\widehat N_{T_{D}({G}(\bar x))}(\nabla {G}(\bar x)u_k) \neq \emptyset$, we must have $\nabla {G}(\bar x)u_k \in T_{D}({G}(\bar x))$.
Thus, taking into account $u_k \to u$, from \eqref{eqn : T and T_T} we get $u_k \in T_{\Gamma}(\bar x)$ for sufficiently large $k$ and $\nabla {G}(\bar x)^Ty^*_k \in \widehat N_{T_{\Gamma}(\bar x)}(u_k)$ follows {by (\ref{add-beta-1})} and \cite[Theorem 6.14]{RoWe98}.
Taking limits as $k\rightarrow \infty$ we conclude $\nabla {G}(\bar x)^Ty^*\in N_{T_{\Gamma}(\bar x)}(u)$, proving $\nabla {G}(\bar x)^TN_{T_{D}({G}(\bar x))}(\nabla {G}(\bar x)u) \subset N_{T_\Gamma(\bar x)}(u)$.
\end{proof}

\section{Second-order variational analysis of disjunctive systems}

In this section, we continue with second-order variational analysis
of the disjunctive system
$\Gamma:=\{x \in \R^n \mv {G}(x)\in {D}\},$
where ${G}$ is now twice continuously differentiable and ${D}$ is polyhedral.
{We will study the domain of the support functions and the connection between the second-order objects $ {\rm d}^2\delta_\Gamma(\bar x;\cdot)(u),\sigma_{T^2_\Gamma(\bar x;u)}(\cdot),\hat\sigma_{T^2_\Gamma(\bar x;u)}(\cdot)$ and the second derivative $\nabla^2 {G}(\bar x)(u,u)$.}
{In the first part, we provide very general results, assuming only that the constraint mapping $x\tto {G}(x)-{D}$ is metrically subregular at $(\xb,0)$ in direction $u$ (MSCQ holds at $\xb$ in direction $u$), and in the second part, we show how everything gets simpler under the directional nondegeneracy
\eqref{EqDirNonDegen00} or its relaxation \eqref{EqSecOrdCRCQ}.}

Given a feasible point $\bar x\in\Gamma$ and a pair $(u,x^*) \in \R^n \times \R^n$,
we denote the set of S- and M- multipliers associated with $(\bar x, x^*)$ in direction $u$, respectively, by
\begin{align*}
\Lambda^s_{x^*} (\bar x;u):=\{y^* \in \widehat{N}_{T_{D}({G}(\xb))}( \nabla {G}(\xb)u)\mv x^*=\nabla {G}(\bar x)^Ty^*\},\\
\Lambda_{x^*} (\bar x;u):=\{y^*\in N_{T_{D}({G}(\bar x))}(\nabla {G}(\bar x)u) \mv x^*=\nabla {G}(\bar x)^Ty^*\}.
 \end{align*}

Consider a direction $u$ belonging to the linearization cone $L_\Gamma(\xb)$ defined by
\[ L_\Gamma(\bar x):=\{u \in \R^n \mv\nabla G(\bar x)u\in T_D(G(\bar x))\}.\]
If {MSCQ holds at $\xb$ in direction $u$, then $u\in T_\Gamma(\xb)$ by Proposition \ref{Prop2.10} and we also get
\begin{equation} \label{EqSecOrdTanGamma2}
T^2_\Gamma(\bar x;u)
=
\{p \in \R^n \mv\nabla {G}(\bar x)p+\nabla^2{G}(\bar x)(u,u)\in T_{T_{D}({G}(\bar x))}(\nabla {G}(\bar x)u)\}
\end{equation}
from Propositions \ref{Prop2.13} and \ref{Prop2.10} (see also \cite{Meh20,MoMoSa20}).}
{\subsection{Subregular systems}}
{We begin with the main result of this section.}

\begin{theorem}\label{ThSOVAGamma}
Let $\bar x\in \Gamma$ and $u \in L_\Gamma(\bar x)$
and suppose that the MSCQ holds at $\xb$ in direction $u$
with the subregularity modulus $\kappa$.
Then the following statements hold:
\begin{enumerate}
\item[\rm (i)] For every $x^*\in \{u\}^\perp$ we have
  \begin{equation}\label{EqSecOrdSubDerGamma}
    {\rm d}^2\delta_\Gamma(\bar x;x^*)(u)=-\sigma_{T^2_\Gamma(\bar x;u)}(x^*).
  \end{equation}
\item[\rm (ii)] We have $\dom \sigma_{T^2_\Gamma(\bar x;u)}=\widehat N^p_\Gamma(\bar x;u) = \widehat N_{T_\Gamma(\bar x)}(u)$. For every
$x^* \in \dom \sigma_{T^2_\Gamma(\bar x;u)}$ {we have $x^*\in \{u\}^\perp$,}  the equality
\eqref{EqSecOrdSubDerGamma} holds and
\begin{equation}\label{EqBndSecSubDeriv1}
  \inf_{y^*\in \Lambda_{x^*} (\bar x;u)\cap \kappa \|x^*\|{\rm cl\,}\mathbb{B}}\skalp{y^*,\nabla^2 {G}(\bar x)(u,u)}\leq {\rm d}^2\delta_\Gamma(\bar x;x^*)(u)\leq
  \sup_{y^*\in  \Lambda_{x^*} (\bar x;u)\cap \kappa \|x^*\|{\rm cl\,}\mathbb{B}}\skalp{y^*,\nabla^2 {G}(\bar x)(u,u)},
  \end{equation}
  \begin{equation}\label{EqBndSecSubDeriv2}
{\rm d}^2\delta_\Gamma(\bar x;x^*)(u) \geq \sup_{y^*\in \Lambda^s_{x^*} (\bar x;u)}\skalp{y^*,\nabla^2 {G}(\bar x)(u,u)}
,
  \end{equation}
  and, moreover, there exists $y^*\in \Lambda_{x^*} (\bar x;u)$ such that
${\rm d}^2\delta_\Gamma(\bar x;x^*)(u)=\skalp{y^*,\nabla^2 {G}(\bar x)(u,u)}$.
\item[\rm (iii)] We have
$\dom \hat\sigma_{T^2_\Gamma(\bar x;u)}\subset \{x^*\,|\,  \Lambda_{x^*} (\bar x;u)\not=\emptyset\}$ and for every
$x^* \in \dom \hat\sigma_{T^2_\Gamma(\bar x;u)}$ it holds
  \begin{equation}\label{EqLowerSuppGamma}
   \inf_{y^*\in \Lambda_{x^*} (\bar x;u)\cap \kappa \|x^*\|{\rm cl\,}\mathbb{B}}\skalp{y^*,\nabla^2 {G}(\bar x)(u,u)}
   \leq
   -\hat \sigma_{T^2_\Gamma(\bar x;u)}(x^*)
   \leq
   \sup_{y^*\in \Lambda_{x^*} (\bar x;u)\cap \kappa \|x^*\|{\rm cl\,}\mathbb{B}}\skalp{y^*,\nabla^2 {G}(\bar x)(u,u)}
  \end{equation}
  and there exists $y^*\in \Lambda_{x^*} (\bar x;u)$ such that
  $ -\hat \sigma_{T^2_\Gamma(\bar x;u)}(x^*)=\skalp{y^*,\nabla^2 {G}(\bar x)(u,u)}$.
 Moreover, the upper bound in \eqref{EqLowerSuppGamma} is valid for all $x^* \in \R^n$.
 \end{enumerate}
\end{theorem}
\begin{proof}
(i)  Consider $x^*\in \{u\}^\perp$. Taking into account \eqref{indicatorf},
consider sequences $t_k\downarrow 0$ and ${u}_k\to u$ such that $\bar x+t_k{u}_k\in\Gamma$ and
  \begin{equation} {\rm d}^2\delta_\Gamma(\bar x;x^*)(u)=\lim_{k\to\infty}-\frac{2\skalp{x^*,{u}_k}}{t_k}=\lim_{k\to\infty}-\frac{2\skalp{x^*,{u}_k-u}}{t_k}.\label{eqn32}
  \end{equation}
  Since ${G}(\bar x+t_k{u}_k)= {G}(\bar x)+t_k\nabla {G}(\bar x){u}_k+\frac 12 t_k^2(\nabla^2 {G}(\bar x)(u,u)+r_k)\in {D}$ with $r_k \to 0$, we obtain
  \[\nabla {G}(\bar x){u}_k+\frac 12 t_k(\nabla^2 {G}(\bar x)(u,u)+r_k)\in T_{D}({G}(\bar x)).\] Consequently, Proposition \ref{lemexact} yields 
  \begin{equation}\label{rk}
  \nabla {G}(\bar x)\frac{2({u}_k-u)}{t_k}+\nabla^2 {G}(\bar x)(u,u)+r_k\in T_{T_{D}({G}(\bar x))}(\nabla {G}(\bar x)u).
  \end{equation}
From Proposition \ref{Prop2.10} {and (\ref{EqSecOrdTanGamma2})} we get that the mapping
  \begin{equation}\label{EqAuxM}
  \Phi(p):=\nabla {G}(\bar x)p+\nabla^2 {G}(\bar x)(u,u)- T_{T_{D}({G}(\bar x))}(\nabla {G}(\bar x)u)
  \end{equation}
  satisfies
  \begin{equation}\label{EqAuxUnifSubrM}
  \dist{p,\Phi^{-1}(0)}\leq \kappa\dist{0,\Phi(p)}\ \ \forall p\in\R^n.
  \end{equation}
By (\ref{rk}) we have $-r_k\in \Phi(2({u}_k-u)/t_k)$.
Hence, for every $k$ we can find some $p_k\in \Phi^{-1}(0)$ satisfying $\norm{\frac{2({u}_k-u)}{t_k}-p_k}\leq\kappa \norm{r_k}$
and
  \[\nabla {G}(\bar x)p_k+\nabla^2 {G}(\bar x)(u,u)\in T_{T_{D}({G}(\bar x))}(\nabla {G}(\bar x)u)
  \]
  and $p_k \in T^2_\Gamma(\bar x;u)$ follows from (\ref{EqSecOrdTanGamma2}).
  Thus, by definition of the support function, we have $\skalp{x^*,p_k}\leq \sigma_{T^2_\Gamma(\bar x;u)}(x^*)$. Moreover, by (\ref{eqn32}) we have
  \[{\rm d}^2\delta_\Gamma(\bar x;x^*)(u)=\lim_{k\to\infty}-\frac{2\skalp{x^*,u_k-u}}{t_k}=\lim_{k\to\infty}-\skalp{x^*,p_k}\geq -\sigma_{T^2_\Gamma(\bar x;u)}(x^*).\]
  Since the opposite inequality holds by
  Proposition \ref{LemBasicSecOrdSubDer}(iv),  
   \eqref{EqSecOrdSubDerGamma} is established.

(ii) We have already shown in Proposition \ref{PropDirProxNormalCone} the inclusion $\widehat N^p_\Gamma(\bar x;u)\subset \widehat N_{T_\Gamma(\bar x)}(u)$.
 Further, since $u \in T_\Gamma(\bar x)$, by Proposition \ref{Prop2.10} we get
 $\nabla {G}(\bar x)u \in T_{D}({G}(\bar x))$ and
 $T_{T_{D}({G}(\bar x))}(\nabla {G}(\bar x)u) \neq \emptyset$ follows.
 This, however, ensures $T^2_\Gamma(\bar x;u)\neq \emptyset$ {due to (\ref{EqAuxUnifSubrM}) and $\Phi^{-1}(0)=T^2_\Gamma(\bar x; u)$ by   (\ref{EqSecOrdTanGamma2})}.  Hence
 $\sigma_{T^2_\Gamma(\bar x;u)}(x^*) > -\infty$ holds for all $x^* \in \R^n$.
Now let $x^*\in \widehat N_{T_\Gamma(\bar x)}(u)$
and we first show that
$x^* \in \dom \sigma_{T^2_\Gamma(\bar x;u)}$.
Assume on the contrary that $\sigma_{T^2_\Gamma(\bar x;u)}(x^*)=\infty$ and consider a sequence $p_k\in T^2_\Gamma(\bar x;u)$ with $\skalp{x^*,p_k}\to\infty$ as $k\to\infty$.
  By \eqref{EqSecOrdTanGamma2} we get
  \begin{equation}
  \nabla {G}(\bar x)p_k+\nabla^2{G}(\bar x)(u,u)\in T_{T_{D}({G}(\bar x))}(\nabla {G}(\bar x)u).  \label{eqn37}
  \end{equation}
  The mapping
  $$p\tto \nabla {G}(\bar x)p-T_{T_{D}({G}(\bar x))}(\nabla {G}(\bar x)u)$$
  is polyhedral, i.e., its graph is a polyhedral set, and is therefore metrically subregular at $(0,0)$ by Robinsons's result \cite{Rob81}.
  Since its graph is also a closed cone, Proposition \ref{closedcone} yields the existence of $\kappa'>0$
  such that for every $k$ we can find some $\tilde p_k$ with $\nabla {G}(\bar x)\tilde p_k\in T_{T_{D}({G}(\bar x))}(\nabla {G}(\bar x)u)$ and
  \[\norm{\tilde p_k-p_k}\leq \kappa'\dist{\nabla {G}(\bar x)p_k,T_{T_{D}({G}(\bar x))}(\nabla {G}(\bar x)u)}\leq \kappa'\norm{\nabla^2{G}(\bar x)(u,u)},\]
  where the second inequality follows from (\ref{eqn37}).
 Since $\nabla {G}(\bar x)\tilde p_k\in T_{T_{D}({G}(\bar x))}(\nabla {G}(\bar x)u)$,
 by Proposition \ref{Prop2.10} we obtain $\tilde p_k\in T_{T_\Gamma(\bar x)}(u)$,
 implying $\skalp{x^*,\tilde p_k}\leq 0$ due to $x^*\in \widehat N_{T_\Gamma(\bar x)}(u)=(T_{T_\Gamma(\bar x)}(u))^\circ$.
 This, however, contradicts the assumption that
 $\skalp{x^*, p_k}\to\infty$ as $k\to \infty$ since the sequence $\{\tilde p_k-p_k\}$ is bounded, showing $x^* \in \dom \sigma_{T^2_\Gamma(\bar x;u)}$.

 Consider now $x^* \in \dom \sigma_{T^2_\Gamma(\bar x;u)}$.
 Note that in order to show that $x^* \in \widehat N^p_\Gamma(\bar x;u)$,
 it suffices to prove $x^*\in \{u\}^\perp$, since then we get
 \eqref{EqSecOrdSubDerGamma} and
Proposition \ref{LemBasicSecOrdSubDer}(iii) gives the claim.
 As we will see, however, $x^*\in \{u\}^\perp$ comes as a by-product
 of the following arguments.

 Since $T_{T_{D}({G}(\bar x))}(\nabla {G}(\bar x)u)$ is a polyhedral cone, it can be written as the union of finitely many convex polyhedral cones, say $K_i$, $i=1,\ldots,s$, and therefore
  \[\sigma_{T^2_\Gamma(\bar x;u)}(x^*)=\max_{i=1,\ldots,s}\sup_{p}\{\skalp{x^*,p}\mv \nabla {G}(\bar x)p+\nabla^2 {G}(\bar x)(u,u)\in K_i\}.\]
  Taking into account that $\sigma_{T^2_\Gamma(\bar x;u)}(x^*)$ is finite,
  every linear program $\sup_{p}\{\skalp{x^*,p}\mv \nabla {G}(\bar x)p+\nabla^2 {G}(\bar x)(u,u)\in K_i\}$ is either infeasible, resulting in the optimal value $-\infty$ or has a finite optimal value and this optimal value is attained, see, e.g., \cite[Theorem 2.198]{BonSh00}.
  Hence, the program
  \begin{equation}\label{maxprog}
  \max \skalp{x^*,p}\mbox{ subject to }\nabla {G}(\bar x)p+\nabla^2{G}(\bar x)(u,u)\in T_{T_{D}({G}(\bar x))}(\nabla {G}(\bar x)u)
  \end{equation}
  has an optimal solution $\bar p$ and $\sigma_{T^2_\Gamma(\bar x;u)}(x^*)=\skalp{x^*,\bar p}$ follows.
  The corresponding constraint mapping is precisely $\Phi$ from \eqref{EqAuxM}
  and it is metrically subregular at $(\bar p,0)$ by \eqref{EqAuxUnifSubrM}.
  Thus, by \cite[Theorem 3]{GfrYe}, there exists a multiplier $y^*$ fulfilling the first-order optimality conditions
 \begin{eqnarray}
 && -x^*+\nabla {G}(\bar x)^Ty^*=0, \ {\norm{y^*}\leq \kappa \norm{x^*}}\label{eqn1},\\
 &&  y^*\in N_{T_{T_{D}({G}(\bar x))}(\nabla {G}(\bar x)u)}(\nabla {G}(\bar x)\bar p+\nabla^2{G}(\bar x)(u,u)) \subset N_{T_{D}({G}(\bar x))}(\nabla {G}(\bar x)u),\label{eqn2}
 \end{eqnarray}
 where in (\ref{eqn2}) we used \cite[Proposition 6.27(a)]{RoWe98}.
Particularly, since $T_{D}({G}(\bar x))$ and
$T_{T_{D}({G}(\bar x))}(\nabla {G}(\bar x)u)$ are cones, we conclude
 \[
 \skalp{y^*, \nabla {G}(\bar x)u}=0
 \quad \textrm{and} \quad
 \skalp{y^*,\nabla {G}(\bar x)\bar p+\nabla^2 {G}(\bar x)(u,u)}=0.
   \]
This means, however $ \skalp{x^*,u} = \skalp{\nabla {G}(\bar x)^Ty^*, u}=\skalp{y^*, \nabla {G}(\bar x)u}=0$ and
\begin{eqnarray}
&&  \sigma_{T^2_\Gamma(\bar x;u)}(x^*) = \skalp{x^*,\bar p}=-\skalp{y^*,\nabla^2 {G}(\bar x)(u,u)}\label{equation}
\end{eqnarray}
and we indeed get \eqref{EqSecOrdSubDerGamma} as claimed.
Moreover, conditions (\ref{eqn1}) and (\ref{eqn2}) ensure that $y^*\in \Lambda_{x^*}(\bar x;u)$ and so \eqref{EqBndSecSubDeriv1}
follows from \eqref{EqSecOrdSubDerGamma}.

  In order to show \eqref{EqBndSecSubDeriv2}, consider $y^*\in \Lambda^s_{x^*} (\bar x;u)$, i.e.,
  $$y^*\in \widehat N_{T_{D}({G}(\bar x))}(\nabla {G}(\bar x)u)
  =[T_{T_{D}({G}(\bar x))}(\nabla {G}(\bar x)u)]^\circ, \quad x^*=\nabla {G}(\bar x)^T y^*.$$
  Since $\bar p$ is an optimal solution of program (\ref{maxprog}), it is feasible, i.e.,
  $$\nabla {G}(\bar x)\bar p+\nabla^2{G}(\bar x)(u,u)\in T_{T_{D}({G}(\bar x))}(\nabla {G}(\bar x)u).$$
  It follows that $ \skalp{y^*,\nabla {G}(\bar x)\bar p+\nabla^2 {G}(\bar x)(u,u)}\leq 0$, showing $\sigma_{T^2_\Gamma(\bar x;u)}(x^*)=\skalp{x^*,\bar p}\leq -\skalp{y^*,\nabla^2 {G}(\bar x)(u,u)}$.
  Again, \eqref{EqBndSecSubDeriv2} follows from \eqref{EqSecOrdSubDerGamma}. {Finally by (\ref{EqSecOrdSubDerGamma}) and (\ref{equation}), we obtain the last conclusion of (ii).}

  (iii) Let $x^*$ satisfy
  $$\hat\sigma_{T^2_\Gamma(\bar x;u)}(x^*):=\liminf_{\tilde x^*\to x^*}\inf_{p'}\{ \skalp{\tilde x^*,p'} \mid \tilde x^*\in \widehat N_{T^2_\Gamma(\bar x;u)}(p')\} <\infty.$$
  Consider the sequences $x_k^*\to x^*$ and $p_k$ such that $x_k^*\in \widehat N_{T_\Gamma^2(\bar x;u)}(p_k)$ and $\hat\sigma_{T^2_\Gamma(\bar x;u)}(x^*)=\lim_{k\to\infty}\skalp{x_k^*,p_k}$.
  By (\ref{EqSecOrdTanGamma2}), we have $T^2_\Gamma(\bar x;u)=\Phi^{-1}(0)$ where the mapping $\Phi$ is given by \eqref{EqAuxM}.
  Since $\Phi$ is metrically subregular with modulus $\kappa$ at $(p_k,0)$ and $x_k^*\in \widehat N_{T_\Gamma^2(\bar x;u)}(p_k)\subset N_{T_\Gamma^2(\bar x;u)}(p_k)=N_{\Phi^{-1}(0)}(p_k)$,
  \cite[Proposition 4.1]{GfrOut16} yields the existence of some $y^*_k$ satisfying $\norm{y^*_k}\leq\kappa\norm{x_k^*}$ and  $(x_k^*,-y^*_k)\in N_{\gph \Phi}(p_k,0)$.
 Applying the change of coordinates formula (see e.g. \cite[Exercise 6.7]{RoWe98}) to $N_{\gph \Phi}(p_k,0)$, we obtain
  \begin{eqnarray}
&&  x_k^*=\nabla {G}(\bar x)^Ty^*_k, \label{eqn41}\\
&& y^*_k\in N_{T_{T_{D}({G}(\bar x))}(\nabla {G}(\bar x)u)}(\nabla {G}(\bar x)p_k+\nabla^2{G}(\bar x)(u,u)) \subset N_{T_{D}({G}(\bar x))}(\nabla {G}(\bar x)u),\label{eqn42}
\end{eqnarray}
   taking into account \cite[Proposition 6.27(a)]{RoWe98} as before. 
   Since $T_{T_{D}({G}(\bar x))}(\nabla {G}(\bar x)u)$ is a cone, by (\ref{eqn42}) we have $\langle y^*_k,  \nabla {G}(\bar x)p_k+\nabla^2{G}(\bar x)(u,u)\rangle =0$, which together with (\ref{eqn41}) implies that
    \begin{equation}
    \skalp{x_k^*,p_k}=-\skalp{y^*_k,\nabla^2 {G}(\bar x)(u,u)}.\label{eqn44}
    \end{equation}
    Since the sequence
    $y^*_k$ is bounded due to $\norm{y^*_k}\leq\kappa\norm{x_k^*}$, we can assume that it converges to some $y^*$ with
    $\norm{y^*}\leq \kappa \norm{x^*}$.
    Taking limits in (\ref{eqn41}), (\ref{eqn42})
    and (\ref{eqn44}), we obtain $y^*\in \Lambda_{x^*} (\bar x;u)$ and $\hat\sigma_{T^2_\Gamma(\bar x;u)}(x^*)=-\skalp{y^*,\nabla^2 {G}(\bar x)(u,u)}$, proving \eqref{EqLowerSuppGamma}.
    The upper bound in \eqref{EqLowerSuppGamma} is obviously valid if $\hat\sigma_{T^2_\Gamma(\bar x;u)}(x^*) = \infty$.
\end{proof}

\begin{remark}
   Recall that by Proposition \ref{Prop2.13} we have
   \[N_{T_{D}(z)}(w)={N}_{D}(z; w),\ {\forall} z\in D,\ w\in T_D(z).\]
   Applying Theorem \ref{ThSOVAGamma}(ii) with $G$ being the identity mapping and $\Gamma=D$ yields the counterpart
   \[\widehat N_{T_{D}(z)}(w)=\widehat {N}_{D}^p(z; w),\ 
   {\forall} z\in D,\ w\in T_D(z).\]
\end{remark}

Note that the bounds for the second subderivative and the lower generalized support function have the same structure, the only difference being the range of validity.
Further note that although the inclusion
$\dom \hat\sigma_{T^2_\Gamma(\bar x;u)}\subset \{x^*\mv \Lambda_{x^*} (\bar x;u)\not=\emptyset\}$ holds, it might be strict in general.
{However, the equality can be obtained under directional nondegeneracy condition, as shown in Corollary
\ref{CorDirNonDegen}(i) below.}

\begin{remark}
Inspired by \cite[Proposition 5.4]{MoMoSa20}, we further show that in (\ref{EqBndSecSubDeriv2}) the {supremum} over $\Lambda^s_{x^*} (\bar x;u)$ provides a tight lower 
{bound} for ${\rm d}^2\delta_\Gamma(\bar x; x^*)(u)$ from the point of view of {weak duality.}

In fact
\begin{eqnarray*}
{\rm d}^2\delta_\Gamma(\bar x; x^*)(u)&=&-\sigma_{T^2_\Gamma(\bar x;u)}(x^*)\\
&=& \min\limits_{p}\{ \langle p, -x^* \rangle| \nabla {G}(\bar x)p+\nabla^2 {G}(\bar x)(u, u)\in T^2_{D}({G}(\bar x); \nabla {G}(\bar x)u)\}\\
&=& \min\limits_{p} \langle p, -x^* \rangle+\delta_{T^2_{D}({G}(\bar x); \nabla {G}(\bar x)u)} \big(\nabla {G}(\bar x)p+\nabla^2 {G}(\bar x)(u, u)\big).
\end{eqnarray*}
The {conjugate dual} problem of the above minimization problem  takes the form
\[
\max\limits_{y^*}\big\{\min \limits_{p} \langle p, -x^* \rangle+\langle y^*, \nabla {G}(\bar x)p
+\nabla^2 {G}(\bar x)(u, u) \rangle
-\sigma_{T^2_{D}({G}(\bar x);\nabla {G}(\bar x)u)}(y^*)    \big\};
\] see e.g., \cite[(2.298)]{BonSh00}.
Note that by Proposition \ref{Prop2.13},
\[
\sigma_{T^2_{D}({G}(\bar x); \nabla {G}(\bar x)u)}(y^*) =\sigma_{T_{T_{D}({G}(\bar x))}(\nabla {G}(\bar x)u)}(y^*)
=\left\{\begin{array}{ll}
0  \ \ &\ y^*\in \widehat N_{T_{D}({G}(\bar x))}(\nabla {G}(\bar x)u)\\
+\infty \ \ & otherwise.
\end{array}  \right.
\]
Hence, the dual problem can be rewritten equivalently as
\[
\begin{array}{ll}
 \sup_{y^*}   & \langle y^*, \nabla^2 {G}(\bar x)(u, u)\rangle\\
  {\rm s.t.}  & x^*=\nabla {G}(\bar x)^Ty^*,\ y^*\in \widehat N_{T_{D}({G}(\bar x))}(\nabla {G}(\bar x)u)
\end{array}
\Longleftrightarrow \sup_{y^*\in \Lambda^s_{x^*}(\bar x;u)}\skalp{y^*,\nabla^2 {G}(\bar x)(u,u)}.
\]
Consequently, by the weak duality {\cite{BonSh00}}, we have
$${\rm d}^2\delta_\Gamma(\bar x; x^*)(u)\geq  \sup_{y^*\in \Lambda^s_{x^*} (\bar x;u)}\skalp{y^*,\nabla^2 {G}(\bar x)(u,u)}.$$
\end{remark}
{According to Theorem \ref{ThSOVAGamma}(ii), 
we observe that when the directional S- and M- multipliers coincide, the following equality holds:
\begin{equation*} {\rm d}^2\delta_\Gamma(\bar x;x^*)(u)=
  \max_{y^*\in  \Lambda_{x^*} (\bar x;u)}\skalp{y^*,\nabla^2 {G}(\bar x)(u,u)}. \label{eqn58}\end{equation*}
Now  we further require that ${D}$ is convex polyhedral. Consider
the critical cone defined by $K(\bar x; x^*):=\{u \in \{x^*\}^\perp \mv \nabla {G}(\bar x)u\in T_{D}({G}(\bar x))\}$.   Then for any critical direction $u\in K(\bar x; x^*)$, it turns out that all S-and M-multipliers coincide with the set of (nondirectional) multipliers
$\Lambda_{x^*}(\bar x):=\{y^*\in N_{D}({G}(\bar x)) \mv x^*=\nabla {G}(\bar x)^T y^*\}$. Consequently in the following corollary by using  Theorem \ref{ThSOVAGamma}(ii)
 we can recover the result  \cite[Exercise 13.17]{RoWe98} under a weaker condition with the metric regularity replaced by the metric subregularity.
It should be noted that the following result can also be obtained from \cite[Example 3.4 and Theorem 5.6]{MoMoSa20}
by specifying the general convex set considered therein to be convex polyhedral.}

\begin{corollary} \label{Cor:convex_polyhedral}
Assume that ${D}$ is convex polyhedral and that the MSCQ holds at $\xb \in \Gamma={G}^{-1}(D)$.
Then for any $x^*\in N_\Gamma(\bar x)$ and any $u \in \R^n$ one has
\begin{equation}\label{D-convex-1}
{\rm d}^2\delta_{\Gamma}(\bar x; x^*)(u)=\delta_{K(\bar x; x^*)}(u)+\max\limits_{y^*\in {\Lambda_{x^*}(\bar x)}}\skalp{y^*,\nabla^2{G}(\bar x)(u,u)}.
\end{equation}
\end{corollary}

\begin{proof}
 According to Proposition \ref{Prop2.10} together with \cite[Theorem 6.14]{RoWe98} we have
\begin{equation*}\label{D-convex-2}
 N_\Gamma (\bar x) \subset \nabla {G}(\bar x)^T N_{D}({G}(\bar x))=\nabla {G}(\bar x)^T \widehat N_{D}({G}(\bar x))
 \subset \widehat N_\Gamma(\bar x),
\end{equation*}
showing $N_\Gamma (\bar x)=\widehat N_\Gamma(\bar x)$
as well as ${\Lambda_{x^*}(\bar x)}\neq \emptyset$
for $x^*\in N_\Gamma(\bar x)$.
Further, by {Theorem \ref{ThSOVAGamma}(ii)} we know
\[
\widehat N^p_\Gamma(\bar x)=\widehat N^p_\Gamma(\bar x;0)= \widehat N_{T_\Gamma(\bar x)}(0)=\widehat N_\Gamma(\bar x) = N_\Gamma(\bar x).
\]

If $u\notin K(\bar x; x^*)$, we claim that ${\rm d}^2\delta_\Gamma(\bar x; x^*)(u)=\infty$.
Indeed, we either have $u\notin T_\Gamma(\bar x)$ or
$u\in T_\Gamma(\bar x)$ but $\skalp{x^*,u}\neq 0$,
in which case we must have $\skalp{x^*,u}<0$ since
 $\skalp{x^*,u}\leq 0$ due to $x^*\in N_\Gamma(\bar x)=\widehat N_\Gamma(\bar x)=(T_\Gamma(\bar x))^\circ$.
 In either case, however, Proposition \ref{LemBasicSecOrdSubDer}(i) yields
 ${\rm d}^2\delta_\Gamma(\bar x; x^*)(u)=\infty$.
  Since ${\Lambda_{x^*}(\bar x)}\neq \emptyset$ as shown above, we get $\sup_{y^*\in {\Lambda_{x^*}(\bar x)}}\skalp{y^*,\nabla^2{G}(\bar x)(u,u)}>-\infty$ and thus (\ref{D-convex-1}) holds.

 If $u\in K(\bar x; x^*)$, then
 \begin{eqnarray*}
  \Lambda_{x^*}(\bar x;u)=\Lambda^s_{x^*}(\bar x; u)
  &=&\{y^*|\, y^*\in \widehat N_{T_{D}({G}(\bar x))}(\nabla {G}(\bar x)u), \ x^*=\nabla {G}(\bar x)^Ty^*\} \nonumber \\
 &=&  \{y^*|\, y^*\in N_{D}({G}(\bar x)),\ \skalp{y^*,\nabla {G}(\bar x)u}=0,\  x^*=\nabla {G}(\bar x)^Ty^*\} \nonumber \\
 &=&  \{y^*|\, y^*\in N_{D}({G}(\bar x)),\ \skalp{x^*,u}=0,\  x^*=\nabla {G}(\bar x)^Ty^*\} \nonumber \\
&=&  \{y^*|\, y^*\in N_{D}({G}(\bar x)),\ x^*=\nabla {G}(\bar x)^Ty^*\} \nonumber \\
&=& {\Lambda_{x^*}(\bar x)}, \label{D-convex-3}
 \end{eqnarray*}
 where the first equality is due to the convexity of ${D}$ (see (\ref{convexcase})) and { the fifth} equality comes from the fact that $\skalp{x^*,u}=0$ holds automatically as $u\in K(\bar x; x^*)$.
{Since $x^*\in N_\Gamma(\bar x)=\widehat N^p_\Gamma(\bar x)\subset \Np_\Gamma(\bar x; u)$ and $\skalp{x^*,u}=0$, we have
  $x^*\in \widehat N_\Gamma^p(\bar x; u)$.} It then follows from (\ref{EqBndSecSubDeriv1}) and (\ref{EqBndSecSubDeriv2}) that
  \[
  {\rm d}^2\delta_\Gamma(\bar x;x^*)(u)=\sup_{y^*\in {\Lambda_{x^*}(\bar x)}}\skalp{y^*,\nabla^2 {G}(\bar x)(u,u)}=
  \sup_{y^*\in  {\Lambda_{x^*}(\bar x)} \cap \kappa \norm{x^*} {\rm cl\,} \mathbb{B}}\skalp{y^*,\nabla^2 {G}(\bar x)(u,u)}.
  \]
  Since ${\Lambda_{x^*}(\bar x)} \cap \kappa \norm{x^*}{\rm cl\,} \mathbb{B}$ is a compact set, the supremum can be attained
  and so it can be {replaced} by the maximum.
  This completes the proof.
  \end{proof}
{\subsection{Nondegenerate systems}}
{As we have seen in Corollary \ref{Cor:convex_polyhedral}, the results from Theorem \ref{ThSOVAGamma} get considerably simpler
if set ${D}$ is convex polyhedral.
Here we continue with simplifications, but we keep ${D}$ arbitrary polyhedral and strengthen the assumptions on constraints instead.
We start with the condition \eqref{EqSecOrdCRCQ}, which does not yield uniqueness of $y^*\in \Lambda_{x^*} (\bar x;u)$,
but it implies that the value $\skalp{y^*,\nabla^2 {G}(\bar x)(u,u)}$ is the same for all the multipliers, making
the lower and upper bounds in \eqref{EqBndSecSubDeriv1} and \eqref{EqLowerSuppGamma} equal.
Since (\ref{EqSecOrdCRCQ}) is in general weaker than the directional nondegeneracy condition (\ref{EqDirNonDegen1}),
we refer to it as the generalized directional nondegeneracy condition.
Moreover, we present the second-order tangent cone $T_\Gamma^2(\bar x;u)$ as a translation of a cone in the form (\ref{form}).}


\begin{proposition}\label{lem4.2}
Let $\bar x\in \Gamma$ and $u \in L_\Gamma(\xb)$ 
and suppose that the MSCQ holds at $\xb$ in direction $u$.
  Assume that the generalized directional nondegeneracy condition
  \begin{equation}
    \label{EqSecOrdCRCQ}\nabla {G}(\bar x)^Ty^*=0,\ y^*\in { {\rm span\,}{N_{D}({G}(\bar x); \nabla {G}(\bar x)u)} }\ \Longrightarrow\ \skalp{y^*,\nabla^2 {G}(\bar x)(u,u)}=0
  \end{equation} holds.
  Then {for any $x^*$ the quantity  $\skalp{y^*,\nabla^2 {G}(\bar x)(u,u)}$ is the same for all $y^*\in \Lambda_{x^*} (\bar x;u)$, i.e.,}
  \begin{equation}\label{EqSup=Inf}\sup_{y^*\in \Lambda_{x^*} (\bar x;u)}\skalp{y^*,\nabla^2 {G}(\bar x)(u,u)}=\inf_{y^*\in \Lambda_{x^*} (\bar x;u)}\skalp{y^*,\nabla^2 {G}(\bar x)(u,u)}\ \ \forall x^*:\Lambda_{x^*} (\bar x;u)\not=\emptyset,
  \end{equation}
  and there is some $p_0$ satisfying
  \begin{equation}\label{EqGammaV0}
    \nabla {G}(\bar x)p_0+\nabla^2 {G}(\bar x)(u,u)\in \lin{T_{T_{D}({G}(\bar x))}(\nabla {G}(\bar x)u)}.
  \end{equation}
  For every such $p_0$ we have the representation
  \begin{equation}
  T_\Gamma^2(\bar x;u)=p_0+ K_{\bar x;u}, \label{form}
  \end{equation}
  where $K_{\bar x;u}:=\{p\mv \nabla {G}(\bar x)p\in T_{T_{D}({G}(\bar x))}(\nabla {G}(\bar x)u)\}$ is a cone.
\end{proposition}

\begin{proof}
Consider $y^*_1,y^*_2\in \Lambda_{x^*} (\bar x;u)$. By definition we have $\nabla {G}(\bar x)^T(y^*_1-y^*_2)=x^*-x^*=0$ and {$y^*_1-y^*_2\in N_{D}({G}(\bar x); \nabla {G}(\bar x)u)-N_{D}({G}(\bar x); \nabla {G}(\bar x)u)\subset  {\rm span\,}{N_{D}({G}(\bar x); \nabla {G}(\bar x)u)}$.} It follows by (\ref{EqSecOrdCRCQ}) that
$$\skalp{y^*_1-y^*_2,\nabla^2{G}(\bar x)(u,u)}=0.$$
Consequently, \eqref{EqSup=Inf} holds.
From \eqref{EqSecOrdCRCQ} and Proposition \ref{lem3.1} we infer
 \begin{equation*}
  \nabla^2 {G}(\bar x)(u,u)
  \in
  \big(\ker \nabla {G}(\bar x)^T\cap {\rm span\,}{ N_{D}({G}(\bar x); \nabla {G}(\bar x)u)
}\big)^\perp
\subset
\nabla {G}(\bar x)\R^n+ \lin{T_{T_{D}({G}(\bar x))}(\nabla {G}(\bar x)u)},
 \end{equation*}
 which yields the existence of $p_0$ satisfying \eqref{EqGammaV0}.
  In fact, by (\ref{EqSecOrdTanGamma2}) we know $p_0\in T^2_\Gamma(\bar x; u)$.

 For every $p\in K_{\bar x:u}$ we have
 \[\nabla {G}(\bar x)(p_0+p)+\nabla^2{G}(\bar x)(u,u)\in T_{T_{D}({G}(\bar x))}(\nabla {G}(\bar x)u) + \lin{T_{T_{D}({G}(\bar x))}(\nabla {G}(\bar x)u)}= T_{T_{D}({G}(\bar x))}(\nabla {G}(\bar x)u),\]
 showing $p_0+p\in T^2_\Gamma(\bar x;u)$ by \eqref{EqSecOrdTanGamma2}.
 On the other hand, if $p\in T^2_\Gamma(\bar x;u)$, then by (\ref{EqSecOrdTanGamma2}) we obtain
 \[\nabla {G}(\bar x)(p-p_0)\in  T_{T_{D}({G}(\bar x))}(\nabla {G}(\bar x)u) - \lin{T_{T_{D}({G}(\bar x))}(\nabla {G}(\bar x)u)}=T_{T_{D}({G}(\bar x))}(\nabla {G}(\bar x)u),\]
 i.e., $p-p_0\in K_{\bar x;u}$.
 This verifies $T_\Gamma^2(\bar x;u)=p_0+ K_{\bar x;u}$ and the proof is complete.
\end{proof}

{Combining Proposition \ref{lem4.2} and Proposition \ref{PropHatSigma} together yields the following result.}

\begin{corollary}\label{Cor4.3}
Let $\bar x\in \Gamma$ and $u \in L_\Gamma(\xb)$
and suppose that the MSCQ holds at $\xb$ in direction $u$.
If {the generalized directional nondegeneracy condition} \eqref{EqSecOrdCRCQ} is fulfilled, then
  \begin{align}
    &{\rm d}^2\delta_\Gamma(\bar x;x^*)(u)=-\sigma_{T^2_\Gamma(\bar x;u)}(x^*)=\skalp{y^*,\nabla^2 {G}(\bar x)(u,u)}\ \ \forall x^*\in \dom \sigma_{T^2_\Gamma(\bar x;u)}=\widehat N^p_\Gamma(\bar x;u), \label{eqn51}\\
    &-\hat\sigma_{T^2_\Gamma(\bar x;u)}(x^*)=\skalp{y^*,\nabla^2 {G}(\bar x)(u,u)}=-{\langle  x^*, p_0 \rangle} \ \ \forall x^*\in \dom \hat\sigma_{T^2_\Gamma(\bar x;u)}= N_{K_{\bar x;u}}(0),\label{eqn52}
  \end{align}
{where $y^*$ is an arbitrary element from $\Lambda_{x^*} (\bar x;u)$ and $p_0$ is an arbitrary vector satisfying $(\ref{EqGammaV0})$, respectively.}
\end{corollary}

\begin{proof}
Note that (\ref{eqn51}) holds by Theorem \ref{ThSOVAGamma}(i)(ii)  and Proposition \ref{lem4.2}.
Let $x^*\in \dom \hat\sigma_{T^2_\Gamma(\bar x;u)}$.
Then, by Theorem \ref{ThSOVAGamma}(iii) and Proposition \ref{lem4.2}, for any $ y^*\in   \Lambda_{x^*} (\bar x;u)$
we have
$$-\hat\sigma_{T^2_\Gamma(\bar x;u)}(x^*)=\skalp{y^*,\nabla^2 {G}(\bar x)(u,u)}.$$
Moreover, by Proposition \ref{lem4.2} there exists $p_0$ satisfying $(\ref{EqGammaV0})$ such that $T_\Gamma^2(\bar x;u)=p_0+ K_{\bar x;u}.$ Hence, (\ref{eqn52}) follows from Proposition \ref{PropHatSigma}.
\end{proof}

\begin{remark}
  Under the assumptions of Corollary \ref{Cor4.3}, we have
  \[{\rm d}^2\delta_\Gamma(\bar x;x^*)(u)=-\sigma_{T^2_\Gamma(\bar x;u)}(x^*)=-\hat\sigma_{T^2_\Gamma(\bar x;u)}(x^*)
  \quad
  \forall x^*\in \dom \sigma_{T^2_\Gamma(\bar x;u)}\subset \dom \hat\sigma_{T^2_\Gamma(\bar x;u)}.\]
  Thus, whenever ${\rm d}^2\delta_\Gamma(\bar x;x^*)(u)$ differs from $-\hat\sigma_{T^2_\Gamma(\bar x;u)}(x^*)$ there holds $\sigma_{T^2_\Gamma(\bar x;u)}(x^*)=\infty$.
\end{remark}

In general we only know an inclusion
\[N_{K_{\bar x;u}}(0)\subset \nabla {G}(\bar x)^TN_{T_{T_{D}({G}(\bar x))}(\nabla {G}(\bar x)u)}(0)=\nabla {G}(\bar x)^TN_{T_{D}({G}(\bar x))}(\nabla {G}(\bar x)u){=\{x^*\,|\, \Lambda_{x^*} (\bar x;u)\not=\emptyset\}}.\]
If we strengthen \eqref{EqSecOrdCRCQ} to (\ref{EqDirNonDegen1}), however,
this inclusion holds with equality, {the multipliers become unique,} and we are also able to give an alternative representation of the set $\widehat N^p_\Gamma(\bar x;u)$.

\begin{corollary}\label{CorDirNonDegen}
Let $\bar x\in \Gamma$ and $u \in L_\Gamma(\bar x)$.
  Under the directional nondegeneracy condition (\ref{EqDirNonDegen1}),
 the following statements hold:
  \begin{enumerate}
  \item[\rm (i)]
  We have
  $\dom \sigma_{T^2_\Gamma(\bar x;u)}=\nabla {G}(\bar x)^T\widehat N_{T_{D}({G}(\bar x))}(\nabla {G}(\bar x)u)=\widehat N^p_\Gamma(\bar x;u)$
  and for every
$x^* \in \dom \sigma_{T^2_\Gamma(\bar x;u)}$ the
set $\Lambda^s_{x^*} (\bar x;u)$ is a singleton $\{y^*_0\}$
and ${\rm d}^2\delta_\Gamma(\bar x;x^*)(u)=-\sigma_{T^2_\Gamma(\bar x;u)}(x^*)=\skalp{y^*_0,\nabla^2 {G}(\bar x)(u,u)}$.
\item[\rm (ii)]
We have
$\dom \hat\sigma_{T^2_\Gamma(\bar x;u)}=\nabla {G}(\bar x)^TN_{T_{D}({G}(\bar x))}(\nabla {G}(\bar x)u)=N_\Gamma(\bar x;u)$
and for every
$x^* \in \dom \hat\sigma_{T^2_\Gamma(\bar x;u)}$ the
set $\Lambda_{x^*} (\bar x;u)$ is a singleton $\{y^*_0\}$
and $-\hat\sigma_{T^2_\Gamma(\bar x;u)}(x^*)=\skalp{y^*_0,\nabla^2 {G}(\bar x)(u,u)}$.
  \end{enumerate}
 \end{corollary}
\begin{proof}
Recall that \eqref{EqDirNonDegen1} implies (\ref{FOSCMS}) which further ensures   the MSCQ at $\xb$ in direction $u$.

(i) It follows from Theorem \ref{ThSOVAGamma} and
Proposition \ref{Prop2.10} that
 \begin{equation*}\label{dom-sigma}
 \dom \sigma_{T^2_\Gamma(\bar x;u)}=\widehat N^p_\Gamma(\bar x;u)=\widehat N_{T_\Gamma(\bar x)}(u)=\nabla {G}(\bar x)^T\widehat N_{T_{D}({G}(\bar x))}(\nabla {G}(\bar x)u).
 \end{equation*}
 Particularly, this shows that $\Lambda^s_{x^*} (\bar x;u)\neq \emptyset$.
 Taking into account $\Lambda^s_{x^*} (\bar x;u) \subset \Lambda_{x^*} (\bar x;u)$, the remaining claims
 follow from Corollary \ref{Cor4.3} once we prove
 that $\Lambda_{x^*} (\bar x;u)$ is a singleton in the next step.

 (ii) Note that
 $N_{T_{T_{D}({G}(\bar x))}(\nabla {G}(\bar x)u)}(0)=N_{T_{D}({G}(\bar x))}(\nabla {G}(\bar x)u)=N_{D}({G}(\bar x);\nabla {G}(\bar x)u)$ by Proposition \ref{Prop2.13}.
 Hence, applying Theorem \ref{ThDirNonDegen} to the set $K_{\bar x;u}:=\{p\mv \nabla {G}(\bar x)p\in T_{T_{D}({G}(\bar x))}(\nabla {G}(\bar x)u)\}$ at $p=0$ yields
 \[\dom \hat\sigma_{T^2_\Gamma(\bar x;u)}=N_{K_{\bar x;u}}(0)=\nabla {G}(\xb)^TN_{T_{T_{D}({G}(\bar x))}(\nabla {G}(\bar x)u)}(0)=\nabla {G}(\bar x)^TN_{T_{D}({G}(\bar x))}(\nabla {G}(\bar x)u)=N_\Gamma(\bar x;u),\]
where the first equality comes from (\ref{eqn52}) and the last equality follows from \eqref{Eq(13-14)}.
Let $x^*\in \dom \hat\sigma_{T^2_\Gamma(\bar x;u)}=N_\Gamma(\bar x;u)$.  Suppose that $y^*_1,y^*_2\in \Lambda_{x^*} (\bar x;u).$
Then
$\nabla {G}(\bar x)^T(y^*_1-y^*_2)=0$ and
$$
y^*_1-y^*_2\in N_{D}({G}(\bar x); \nabla {G}(\bar x)u)-N_{D}({G}(\bar x); \nabla {G}(\bar x)u)\subset {\rm span\,}N_{D}({G}(\bar x); \nabla {G}(\bar x)u).
$$
The nondegeneracy condition \eqref{EqDirNonDegen1} yields $y^*_1=y^*_2$, which implies that the set $\Lambda_{x^*} (\bar x;u)$ contains a unique element, say $y^*_0$.
Corollary \ref{Cor4.3} now completes the proof.
\end{proof}
\section{Application: Second-order conditions for disjunctive programs}
{
Let us first reiterate that our ultimate goal is to investigate problems (GP) with set $C$ having the complex structure \eqref{eq:StructureOfC}.
As explained, this has to be postponed until we complete the full analysis in the forthcoming paper.}
In this section, we {only provide a simple} application of our results to disjunctive program defined as
\begin{flalign*} \label{P}
\begin{split}
\mbox{(DP)} \hspace{55mm} \min & \ \  f(x)\\
{\rm s.t.}
& \ \   g(x)\in D,
\end{split}&
\end{flalign*}
where $g$ is twice continuously differentiable and $D$ is polyhedral.
Several classes of interesting mathematical programs of practical interests can be reformulated as a (DP),
including the mathematical program with equilibrium constraints {(MPEC)}  (cf. \cite{GLY13,lpr,Outra,SS00}),  the mathematical program with vanishing constraints (cf. \cite{ak,HK07,hk08}),
the mathematical program with switching constraints (cf. \cite{kms,ly}) and the mathematical program with cardinality constraints (cf. \cite{BS18,cks}). {The discussion on constraint qualifications of disjunctive programming can be found in \cite{Gfr14a,Meh20} and references therein.}

Using the second-order variational analysis of the disjunctive system, we can  now recover the second-order optimality conditions for the (DP) derived by Gfrerer in \cite{Gfr14a}. {Moreover using the calculations for directional S- and M- multiplier sets for MPECs in \cite{Gfr14a}, we can easily obtain the corresponding second-order optimality conditions for MPECs from Theorem \ref{Thm5.6}.}

\begin{theorem}[{\cite[Theorems 3.3 and 3.17]{Gfr14a}}] \label{Thm5.6}
Let $\xb$ be a local optimal solution of the disjunctive program (DP).
Then the following necessary optimality conditions hold:
\begin{itemize}
\item[\rm (i)] For $u \in \mathcal{C}(\xb)$, suppose that $x \tto g(x)-D$ is metrically subregular in direction $u$ at $(\xb,0)$. Then there exists $\lambda \in \Lambda(\xb;u)$ such that
$$\nabla^2_{xx}L(\bar x, \lambda)(u,u) \geq 0.$$
\item[\rm (ii)]
  For $u \in \mathcal{C}(\xb)$, assume that the nondegeneracy condition in direction $u$
  $$\nabla g(\bar x)^Ty^*=0,\ y^*\in {\rm span\,} {N_D(g(\bar x);\nabla g(\bar x)u)}\  \Longrightarrow \ y^*=0 $$
    is fulfilled.  Then $\nabla^2_{xx}L(\bar x, \lambda)(u,u) \geq 0$
    holds with the unique directional S-multiplier $\lambda \in \Lambda^s(\xb;u)$.
\end{itemize}
  Conversely, suppose that $\xb$ is a feasible solution of the disjunctive program (DP). Suppose that for
  each nonzero $u\in \mathcal{C}(\xb)$, there are {$\alpha$ and $\lambda$, not both equal to zero, with $\alpha\geq 0, \lambda \in \widehat{N}_D^p (g(\bar x);\nabla g(\bar x)u)=\widehat{N}_{T_D (g(\bar x))}(\nabla g(\bar x)u)$}, such that
  \[
   \nabla^2_{xx}L^\alpha(\bar x, \lambda)(u,u)>0,
   \]
   then $\xb$ is an essential local minimizer of second order.
\end{theorem}

\begin{proof}
To obtain the necessary optimality conditions,  it suffices to   calculate $\hat{\sigma}_{T_D^2(g(\xb); \nabla g(\xb)u)}(\lambda)$ and apply Theorem \ref{Thm3.1}.
Under the assumptions,  the second-order necessary optimality condition
 (\ref{lambda}) holds with $C:=D$.
 It follows that
 $\hat{\sigma}_{T_D^2(g(\xb); \nabla g(\xb)u)}(\lambda)<\infty$,
 which ensures
 $\lambda \in {\rm dom\,} \hat{\sigma}_{T_D^2(g(\xb); \nabla g(\xb)u)}$.
 Taking $G$ to be the identity mapping in Theorem
 \ref{ThSOVAGamma}(iii),
 we obtain $\Gamma=D$ and
 $\hat{\sigma}_{T_D^2(g(\xb); \nabla g(\xb)u)}(\lambda)=0.$
 Hence, the necessary optimality conditions (a) and (b) hold for (DP).

To obtain the sufficient optimality condition, it suffices to calculate
${\rm d}^2\delta_D(g(\xb); \lambda)(\nabla g(\xb) u)$ and apply Theorem \ref{Theorem5.9} with $C:=D$.
Again, applying Theorem \ref{ThSOVAGamma}(ii) with $G$ being the identity mapping, we have ${\rm d}^2\delta_D(g(\xb); \lambda)(\nabla g(\xb) u)=0$ since $\lambda \in \widehat N^p_D(g(\bar x);\nabla g(\bar x)u)$ and the result follows.
\end{proof}
{Note that we have only used the results from Section 5 for the trivial identity mapping.
Their full potential will be seen when applied to sets of the form \eqref{eq:StructureOfC}.}
\section{Concluding Remarks}\vspace*{-0.1in}
In this paper, we have reviewed the second-order necessary optimality conditions and derived second-order sufficient optimality conditions for the general problem (GP). Since these conditions involve some second-order objects that need to be calculated or estimated, we have conducted second-order variational analysis of disjunctive systems. As an illustration, we have shown that one can recover second-order optimality conditions for disjunctive programs.   In the forthcoming work \cite{BeGfrYeZhouZhang}, using the analysis of disjunctive systems from this paper as a tool, we will develop the variational analysis of the set given by \eqref{eq:StructureOfC},
which will enable us to apply our second-order optimality conditions from Theorem \ref{Thm3.1} and Theorem \ref{Theorem5.9}.\vspace*{-0.1in}

\vskip 7mm
\noindent
{
{\bf Acknowledgements.}
The authors are indebted to the anonymous referees for their valuable suggestions that helped us improve the original presentation of the paper.}

\end{document}